\documentclass[12pt]{amsart}
\usepackage[margin=1in]{geometry}
\usepackage{amsmath,amsfonts,amsthm,amssymb,bbm}
\usepackage{graphicx,color,dsfont}
\usepackage{enumitem}
\usepackage{fourier}

\newtheorem{theorem}{Theorem}
\newtheorem{lemma}{Lemma}
\newtheorem{proposition}{Proposition}

\newtheorem{definition}{Definition}
	
	\newtheorem{corollary}{Corollary}

\newtheorem{claim}{Claim}

 \theoremstyle{definition}
 
 \theoremstyle{remark}

 \numberwithin{equation}{section}

\newcommand{\vertiii}[1]{{\left\vert\kern-0.25ex\left\vert\kern-0.25ex\left\vert #1
    \right\vert\kern-0.25ex\right\vert\kern-0.25ex\right\vert}}

\newcommand{\f}[2]{\frac{#1}{#2}}

\newcommand{\cl}{{\mathcal L}}

\newcommand{\al}{\alpha}

\newcommand{\de}{\delta}

\newcommand{\ka}{\kappa}
\newcommand{\la}{\lambda}

\newcommand{\si}{\sigma}

\newcommand{\vp}{\varphi}

\newcommand{\om}{\omega}
\newcommand{\Om}{\Omega}

\newcommand{\rone}{\mathbb R}

\newcommand{\rthree}{\mathbf R^3}

\newcommand{\dpr}[2]{\langle #1,#2 \rangle}

\newcommand{\eps}{\epsilon}

\newcommand{\ca}{\mathcal A}

\newcommand{\ch}{\mathcal H}

\newcommand{\p}{\partial}

\newcommand{\beq}{\begin{equation}}
\newcommand{\eeq}{\end{equation}}
\newcommand{\beqna}{\begin{eqnarray*}}
\newcommand{\eeqna}{\end{eqnarray*}}
\newcommand{\beqn}{\begin{equation*}}
\newcommand{\eeqn}{\end{equation*}}
\newcommand{\bp}{\begin{proof}}
\newcommand{\ep}{\end{proof}}
\newcommand{\bprop}{\begin{proposition}}
\newcommand{\eprop}{\end{proposition}}
\newcommand{\bt}{\begin{theorem}}
\newcommand{\et}{\end{theorem}}
\newcommand{\bex}{\begin{Example}}
\newcommand{\eex}{\end{Example}}
\newcommand{\bc}{\begin{corollary}}
\newcommand{\ec}{\end{corollary}}
\newcommand{\bcl}{\begin{claim}}
\newcommand{\ecl}{\end{claim}}
\newcommand{\bl}{\begin{lemma}}
\newcommand{\el}{\end{lemma}}

\newcommand{\cj}{{\mathcal J}}
\newcommand{\ci}{{\mathcal I}}

\begin{document}

\title
[Stability of periodic  waves in the cubic derivative NLS and quintic NLS]
{On the stability of periodic  waves  for the  cubic derivative NLS  and the quintic NLS }

 \author{Sevdzhan Hakkaev}
 \author{Milena Stanislavova}
 \author{Atanas Stefanov}

 \address{Sevdzhan Hakkaev, Faculty of Arts and Sciences, Department of Mathematics and Computer Science, Istanbul Aydin University, Istanbul, Turkey\\
 \\
 Faculty of Mathematics and Informatics, Shumen University, Shumen, Bulgaria  }
 \email{s.hakkaev@shu.bg}

 \address{Milena Stanislavova
 	Department of Mathematics, University of Kansas, 1460 Jayhawk
 	Boulevard,  Lawrence KS 66045--7523} \email{stanis@ku.edu}
 \address{Atanas Stefanov
 	Department of Mathematics, University of Kansas, 1460 Jayhawk
 	Boulevard,  Lawrence KS 66045--7523}

 \email{stefanov@ku.edu}

 \thanks{ Milena Stanislavova is partially supported by NSF-DMS,    \# 1614734.
 	Atanas Stefanov acknowledges  partial support  from NSF-DMS,   \# 1908626.}



\date{\today}

\begin{abstract}
 We study the periodic cubic derivative non-linear Schr\"odinger equation (dNLS) and the (focussing) quintic non-linear Schr\"odinger equation (NLS). These are both  $L^2$ critical dispersive models, which exhibit threshold type behavior, when posed  on the line ${\mathbb R}$.

We describe the (three parameter) family of non-vanishing bell-shaped solutions for the periodic problem, in closed form.  The main objective of the paper is to  study their stability with respect to co-periodic perturbations.
We analyze these waves for stability in the framework of the cubic DNLS.  We provide a  criteria for stability, depending on the sign of a scalar quantity. The proof relies on an instability index count, which in turn critically  depends on a  detailed spectral analysis of  a self-adjoint matrix Hill operator. We exhibit a region in parameter space, which produces spectrally stable waves.

We also provide an explicit description of the stability of all bell-shaped traveling waves for the quintic NLS, which turns out to be a two parameter subfamily of the one exhibited for DNLS.  We give a complete description of their stability - as it turns out some are spectrally stable, while other are spectrally unstable, with respect to co-periodic perturbations.

\end{abstract}

\maketitle

\section{Introduction}
  We are interested in  the cubic derivative non-linear Schr\"odinger equation (DNLS) in periodic context.  More specifically, we consider
  \begin{equation}
  \label{dnl:10}
  i q_t+\p_x^2 q+i (|q|^2 q)_x=0.
  \end{equation}
  where $q$ is subject to the  periodic boundary conditions,  $q(-T)=q(T), q_x(-T)=q_x(T)$. This particular model (along with some variations),  was  derived to model polarized Alfven waves in a magnetized plasma, under a constant magnetic field.  As usual, the conserved quantities provide an important threshold information with respect to the well-posedness.  Let us state for the record that, at least for smooth solutions, \eqref{dnl:10} conserves the mean, energy and mass. That is,
  $
  \int_{-T}^T q(t,x) dx=const.$, and
   \begin{eqnarray}
   \label{dnl:14}
   E &=& \int_{-T}^T |q_x|^2dx + \f{3}{2} \Im \int_{-T}^T |q|^{2} q_x  \bar{q} dx+\f{1}{2} \int_{-T}^T |q|^6 dx=const.\\
   \label{dnl:15}
   M &=& \int_{-T}^T |u|^2dx = const.
   \end{eqnarray}

   The basic question,  that one has got to be  immediately interested in,  is the well-posedness of the Cauchy problems \eqref{dnl:10} and \eqref{n:10}. For DNLS,  posed on the real line, local well-posedness is established , for data in $H^s(\rone), s\geq \f{1}{2}$, in \cite{Tak,Tak1}.  This  is  sharp, in the sense that the   data to solution map fails to be Lipschitz in $H^s, s<\f{1}{2}$, \cite{BL,Tak1}. Global solutions may also be constructed, under a specific smallness condition on $\|u_0\|_{L^2}$, \cite{HO, HO1,Oz}.
   There are however intriguing recent results, which make use of the completely integrable structure of \eqref{dnl:10}, \cite{JLPS, LPS1, LPS2, PShima1, PShima} that establish global well-posedness for DNLS, under no smallness requirements, albeit for a.e.  data in weighted Sobolev spaces. Finally, very recently, it was shown in \cite{BG} that the DNLS is globally well-posed for all data in $H^{\f{1}{2}}(\rone)$, without any smallness restrictions. Turning to the existence and the stability of solitary waves for DNLS,
   there has been quite a surge in activity in the last twenty years. In \cite{GW}, the authors have shown the stability of the cubic DNLS solitons,  while \cite{CO} has improved upon their results. Recently, \cite{MTX} has established stability results for sum of two solitary DNLS waves. At this point, we would like to draw the reader's attention to the important work \cite{LSS}. In it, the authors have constructed solitary wave solutions on the line $\rone$ for the generalized DNLS (i.e. with general power $|u|^{2\si} u_x$ in the non-linearity) and they have studied their respective stability. The results obtained therein are about exhaustive and introduce some new methods, that we use ourselves herein.

   For the periodic problem, local well-posedness on the (almost) optimal space $H^s({\mathbb T}),  s>\f{1}{2}$ was established in \cite{Herr}. It is immediate, due to conservation laws, that one can extend such $H^1({\mathbb T})$ solutions to global, under a small $L^2$ data assumption, but the question on whether  large global solutions persist remains open for the periodic cubic DNLS.
 There has been quite an activity recently  on the construction of new solutions to \eqref{dnl:10} in the periodic context, see for example \cite{CZ}  for some new quasi-periodic solutions, using algebro-geometric methods. Also, in \cite{UD}, the authors use the complete integrability of \eqref{dnl:10} to study the spectrum  of the linearized operators  by
  relating it to its Lax spectrum.

   Next, we would like to explore a connection of \eqref{dnl:10} to a related non-linear Schr\"odinger equation, namely
     \begin{equation}
     \label{10}
     i u_t+\p_x^2 u+i |u|^{2} u_x =0, \ \  (t,x)\in \rone_+\times [-T,T]
     \end{equation}
   It is well-known and easy to check  fact is that $q$ is a solution to \eqref{dnl:10} if and only if
   \begin{equation}
   \label{gauge}
   u(t,x)=q(t,x) e^{i \f{1}{2} \int_{-T}^x |q(t,y)|^2 dy}
   \end{equation}
   is a solution\footnote{It is important to observe that in the gauge relation
   	\eqref{gauge}, we have that $|q|=|u|$, so the phase function can be written either with $q$ or $u$ inside of it} of \eqref{10}. Note however that if $q$ is periodic on $[-T,T]$, $u$ is not necessarily periodic on $[-T,T]$. In fact, there is no standard well-posedness theory for \eqref{10} in the periodic context, other than the following - given initial data $u_0$ for \eqref{10}, one  can translate to $q_0(x)=u_0(x)  e^{-i \f{1}{2} \int_{-T}^x |u_0(y)|^2 dy}$ and {\it if  $q_0$ is periodic}, then solve \eqref{dnl:10}. That is, for the Cauchy problem of \eqref{10}, one can solve in the periodic setting for say $H^1[-T,T]$ initial data $u_0: \int_{-T}^T |u_0(x)|^2 dx\in \{0, \pm 4\pi, \pm 8\pi, \ldots\}$, where the solution is recovered through \eqref{gauge}.
   This is something we will need to eventually address.
   On the other hand,  the gauged equation \eqref{10} is better suited for our purposes, as it yields better coordinates for our periodic wave solutions as well as the corresponding linearized problem.

       Another model of interest, which as is turns out is very much
       related  to both \eqref{dnl:10} and \eqref{10},  is the quintic non-linear Schr\"odinger equation (NLS), which takes the form
        \begin{equation}
        \label{n:10}
         \left\{\begin{array}{l}
         iu_t +u_{xx}+b |u|^4 u=0, -T<x<T,\\
         u(0,x)=u_0(x)
         \end{array}
         \right.
        \end{equation}
       subject to the same periodic boundary conditions.  In addition, we consider only the focusing case, so $b>0$.

     The question for local and  global well-posedness for the quintic NLS, \eqref{n:10} is well-understood. To summarize the classical by now results, the local well-posedness,  holds under the assumption $u_0\in H^s, s>\f{1}{2}$, both when the problem is posed on the line $\rone$ or on the torus ${\mathbb T}$.  Such solutions can be extended to global solutions, provided $\|u_0\|_{L^2}$ is small enough. On the other hand, on the line, it is well-known that appropriately chosen initial data,  close to the bell-shaped traveling soliton, will produce a solution, which experience a finite time blow up, that is the soliton experiences instability by blow-up.
       It is not at all clear however, whether or not blow up for large $L^2$ data,  happens in the periodic case.  That is,  it is  an interesting open problem,  whether or not solutions with large $L^2$ data can persist globally. This applies to both DNLS and quintic NLS in the periodic context.

    This is actually one of the motivations behind our work.
     As is well-known, most  of the  dynamical properties of the system,  can be inferred from its periodic waves and the behavior of the Cauchy problem for data close to them.
       Therefore, to understand better the dynamics of the problem, the natural place to start is data close to the periodic waves, in other words their stability.
      In this work, our main goal is to investigate  the stability of the corresponding   waves in the periodic case, which is an outstanding open question in the theory.  We work only with the case of cubic derivative non-linearity, which is physically best motivated, but also because we need explicit formulas for our calculations\footnote{In the work \cite{LSS}, the authors exhibit explicit $sech$ type solutions for all powers $\si$}. We explicitly identify all bell-shaped traveling waves, which turn out to be a rich, three parameter family of explicit solutions. For the  analysis of the matrix Hill operator, our approach mirrors the approach in \cite{LSS}, by relying on the spectral  properties of the scalar linearized operators $\cl_\pm$.  In addition, we use topological methods to establish the expected spectral properties as they are difficult to obtain in a direct manner. Finally, we use the instability index counting theory (instead of the direct Grillakis-Shatah-Strauss approach in \cite{LSS}, where it is somewhat easier to compute the necessary quantities), due to the need to apply topological methods for the  spectral problem \eqref{170} below.

      Next, we give  the description of the waves.

     \subsection{Description of   the  solutions for DNLS}
      We now construct the periodic wave solutions for the DNLS, \eqref{dnl:10}. We look instead for periodic wave solutions  for the gauged equation \eqref{10}, which will then later translate into true DNLS waves via the gauge transformation \eqref{gauge}.

      More precisely, let  $u(x,t)=\Phi(x-ct) e^{i \om t}$ and plug it inside the model \eqref{10}. We obtain the equation
     \begin{equation}
     \label{20}
     -\om \Phi-i c \Phi'+\Phi''+i |\Phi|^{2} \Phi'=0
     \end{equation}
 Further, we use the assignment\
 \begin{equation}
 \label{pel12}
  \Phi(y)=\phi(y)e^{i \theta(y)}, \ \  \theta(y)=\frac{c}{2}y-\frac{1}{4} \int_{-T}^y \phi^{2}(\eta)d\eta+const.,
 \end{equation}
 to reduce the problem to one, where we look  for a real-valued wave $\phi$.  
 	 In terms of $\phi$,  we obtain the following equation
 \begin{equation}
 \label{30}
 -\phi''+(\om-\f{c^2}{4})\phi+\f{c}{2}\phi^{3}-\f{3}{16} \phi^{5} =0\ \ \ -T\leq y\leq T.
 \end{equation}
   Here, it has to be noted that  more general solutions are possible, and this has been done in \cite{CUP}. More specifically, the authors have considered the most general case, by using the ansatz \eqref{pel12}. This introduces a system of two second order ODE for $\theta$ and $\phi$, see equations (5.1) and (5.2) in  \cite{CUP}. Our case corresponds to the case $a=0$ in their setup. On the other hand, it is argued in \cite{CUP} that the solutions with the additional term share the same stability/instability properties with the solutions for which this term set to zero.
 
   In our situation, we concentrate to the case \eqref{pel12}, which forces the relation \eqref{30}. Note that we will be looking for periodic solutions $\phi$ of \eqref{30}, but we should be aware that the resulting function $\Phi(y)=\phi(y)e^{i \theta(y)}$ may not be periodic. The point is that, we will need to translate this back to true periodic waves for the original DNLS equation \eqref{dnl:10} in order to impose the necessary periodic conditions, and we shall do so later on, see \eqref{32} below.
 	
 Turning back to the solution set of \eqref{30}, it is the case that the set of solutions for \eqref{30} is fairly rich as it depends on three independent parameters,  see Proposition \ref{prop:10} below.  Related to this,  the  periodic waves were  described in detail in some similarly general situations, see for example \cite{CP} for   the case of a quadratic nonlinearity.  We will restrict our analysis to the set  of bell-shaped solutions, a notion which we introduce next.

  To this end, we shall need the concept  of a decreasing rearrangement of a function. More precisely, introduce for $\al>0$, $d_f(\al)=|\{x\in (-T,T): |f(x)|>\al\}|$ and for each $t\in (-T,T)$, let $f^*(t):=  \inf \{s>0: d_f(s)<\f{|t|}{2}\}$. Note that $f^*$ is positive, even, decreasing in $[0,T]$. 
\begin{definition}
	\label{defi:10}
	We say that a real-valued function $f\in H^1_{per.}[-T,T]$ is  bell-shaped,  if it coincides with its decreasing rearrangement $f^*$.   Equivalently, $f\in H^1_{per.}[-T,T]$ is bell-shaped, if it is positive and it has a single maximum on $[-T,T]$. 
\end{definition}
{\bf Remark:} Our results will apply equally well to waves $\phi$, so that $|\phi|$ is bell-shaped (i.e. $-\phi$ is bell-shaped), but we shall not dwell on this henceforth.

Going back to \eqref{30} - after multiplying by $\phi$ and  integrating in the equation,
  we get,
 \begin{equation}
 \label{35}
\phi'^2=-\frac{1}{16}\phi^6+\frac{c}{4}\phi^4+\left( \om-\frac{c^2}{4}\right)\phi^2+a,
\end{equation}
where  $a$ is a constant of integration.
We look for solution in the above equation in the form
$\varphi=\phi^2$.   In particular, $\vp$ needs to be a positive function. We get the following equation for $\vp$
\begin{equation}
\label{38}
\varphi'^2=\frac{1}{4} \varphi\left[
-\varphi^3+4c\varphi^2+16\left( \om-\frac{c^2}{4}\right)\varphi+a
\right]=:\frac{1}{4} \varphi\left[a-R(\vp)\right],
\end{equation}
where the cubic polynomial $R$ is given by
$$
R(z)=z^3-4c z^2-16\left(\om-\f{c^2}{4}\right)z.
$$
  Note that as we are looking for bell-shaped and non-vanishing
	solution $\vp$, it must be that all three roots of $R$, $\vp_1\leq 0< \vp_2<\vp_3$ are real, and at least two of them are positive. This is the situation of interest.

We henceforth assume\footnote{Even though, there are certainly interesting solutions for $\om<0$ as well} $\om>0$.
\begin{proposition}(Existence of non-vanishing bell-shaped waves)
	\label{prop:10}
	
	Assume $\om> 0$. We have the following possibilities
	\begin{enumerate}
		\item If  $\om-\f{c^2}{4}>0$, and
		$$
		  \frac{16}{27} \left(\sqrt{c^2+12 \om}+2 c\right) \left(c^2-12 \om-c \sqrt{c^2+12 \om}\right)<a<0,
		  $$
		  then the algebraic equation $R(z)=a$ has three roots $\vp_1<0< \vp_2<\vp_3$, depending on $a,\om,c$ in a smooth manner. As a consequence, \eqref{38} has an unique  bell-shaped solution $\vp$, which satisfies
		  $$
		  \vp(0)=\vp_3, \vp(-T)=\vp(T)=\vp_2.
		  $$
		  Moreover,  we have the explicit formula for the solution
		  \begin{equation}
		  \label{2.5}
		  \phi^2(\xi)=\varphi(\xi)=\frac{\varphi_3(\varphi_2-\varphi_1)+\varphi_1(\varphi_3-\varphi_2)
		  	sn^2\left(
		  	\frac{\xi}{2g}, \kappa\right)}{(\varphi_2-\varphi_1)+(\varphi_3-\varphi_2)sn^2\left(
		  	\frac{\xi}{2g}, \kappa\right)},
		  \end{equation}
		  where
		  \begin{equation}
		  \label{r:24}
		  g=\frac{2}{\sqrt{\varphi_3(\varphi_2-\varphi_1)}}, \ \
		  \kappa^2=\frac{-\varphi_1(\varphi_3-\varphi_2)}{\varphi_3(\varphi_2-\varphi_1)}\in (0,1).
		  \end{equation}
		  \item  If  $\om-\f{c^2}{4}>0$, then for every $a>0$, there is unique solution $\vp_3$ of $a=R(\vp)$, with $\vp_3>0$. As a consequence, there is unique bell-shaped solution $\vp: \vp(0)=\vp_3$, $\vp(-T)=\vp(T)=0$.
		  \item If $\om-\f{c^2}{4}>0$ and
		  $$
		  a\leq \frac{16}{27} \left(\sqrt{c^2+12 \om}+2 c\right) \left(c^2-12 \om-c \sqrt{c^2+12 \om}\right),
		  $$
		  there are no bell-shaped solutions of \eqref{38}.
	\end{enumerate}
	Assume now $\om>0, \om-\f{c^2}{4}<0$. We have the following possibilities:
		\begin{enumerate}
			\item  Assume $c>0$ and
			$$
			\frac{16}{27} \left(\sqrt{c^2+12 \om}+2 c\right) \left(c^2-12 \om-c \sqrt{c^2+12 \om}\right)<a<0,
			$$
			then  the algebraic equation $R(z)=a$ has three roots $\vp_1<0< \vp_2<\vp_3$, depending on $a,\om,c$ in a smooth manner. There exists unique bell-shaped solution, so that $\vp(0)=\vp_3, \vp(-T)=\vp(T)=\vp_2$. The solution $\vp$ is given by the exact same formula \eqref{2.5} as above.
			\item Assume $\om-\f{c^2}{4}<0$. Then, for each $a>0$, there is an unique positive root $\vp_3$. Thus, there is unique bell-shaped solution of \eqref{38}, which satisfies  $\vp(0)=\vp_3, \vp(-T)=\vp(T)=0$.
			\item Assume $c>0$ and
			$$
			a\leq \frac{16}{27} \left(\sqrt{c^2+12 \om}+2 c\right) \left(c^2-12 \om-c \sqrt{c^2+12 \om}\right),
			$$
			then  the equation $R(\vp)=a$ has no positive roots and hence, \eqref{38} has no bell-shaped solutions.
			\item Assume $\om-\f{c^2}{4}<0, c<0$ and $a<0$. Then  the equation $R(\vp)=a$ has no positive roots and hence, \eqref{38} has no bell-shaped solutions.
		\end{enumerate}
\end{proposition}
\begin{proof}
	We note first that the function $R$ has a local minimum at $z=\frac{2}{3} \left(\sqrt{c^2+12 w}+2 c\right)$ and it is the case that
	$$
	R\left(\frac{2}{3} (\sqrt{c^2+12 w}+2 c)\right)= 	\frac{16}{27} \left(\sqrt{c^2+12 \om}+2 c\right) \left(c^2-12 \om-c \sqrt{c^2+12 \om}\right).
	$$
	This implies all the statements about the roots of the algebraic  equation $a=R(z)$. The  existence of solutions made in Proposition \ref{prop:10} follows from an elementary ordinary equations reasoning.
	
	In the cases of three different roots, it remains to establish the formula \eqref{2.5}.
	 If $\varphi_1, \varphi_2, \varphi_3$ are nonzero roots of the
	 polynomial
	 $0=-t^3+4ct^2+16\left( \om-\frac{c^2}{4}\right)t+a$, then the Viet's formulas yield
	 then
	 \begin{equation}
	 \label{63}
	  \left\{ \begin{array}{ll}
	  \varphi_1+\varphi_2+\varphi_3=4c\\
	  \\
	  \varphi_1\varphi_2+\varphi_1\varphi_3+\varphi_2\varphi_3=-16\left(\om-\frac{c^2}{4}\right)\\
	  \\
	  \varphi_1\varphi_2\varphi_3=a.
	  \end{array} \right.
	 \end{equation}
	If
	$\varphi_1<0<\varphi_2<\varphi_3$ and
	$\varphi_2<\varphi<\varphi_3$, we get
	$$\int_{\varphi}^{\varphi_3}{\frac{ds}{\sqrt{s(s-\varphi_1)(s-\varphi_2)(\varphi_3-s)}}}=\frac{1}{2}(\xi-\xi_0)$$
	and the solution $\varphi$ is given by \eqref{2.5}.
	Finally, this solution is $2 T$ periodic, with $\vp(0)=\vp_3$, while $\vp(T)=\vp_2$, where
	\begin{equation}
	\label{61}
	T=2 g K(\ka).
	\end{equation}
\end{proof}
Our next results details the translation to the periodic waves of the DNLS problem \eqref{dnl:10}.
\begin{proposition}
	\label{prop:21}
	Let $\phi$ be the wave constructed in Proposition \ref{prop:10}, that is it is given by \eqref{2.5}. Then,
	\begin{equation}
	\label{dnl:sol}
		q(t,x)=\phi(x-c t) e^{i \om t} e^{i\left(\f{c}{2}(x-ct) -\f{3}{4}\int_{-T}^{x-ct} \phi^2(y) dy+const.\right)},
	\end{equation}
	is a periodic  wave for \eqref{dnl:10}, provided
	\begin{equation}
	\label{32}
	 cT - \frac{3}{4} \int_{-T}^T \phi^{2}(y)dy\in \{0, \pm 2\pi, \pm 4\pi, \ldots\}.
	 \end{equation}
\end{proposition}
\begin{proof}
	We know that $\Phi(x-ct) e^{i \om t} e^{i \theta(x-ct)} $ is a solution of \eqref{10}. We use the gauge transformation to obtain
	$$
	q(t,x)=u(t,x) e^{-i\f{1}{2} \int_{-T}^x |u|^2}=\phi(x-c t) e^{i \om t} e^{i\left(\f{c}{2}(x-ct) -\f{3}{4}\int_{-T}^{x-ct} \phi^2(y) dy+const.\right)},
	$$
	as specified. Finally, it can be easily checked that the function given in \eqref{dnl:sol} is $2T$ periodic, together with its derivative, exactly when \eqref{32} holds true.
\end{proof}
\subsection{The linearized equations for DNLS}
\label{sec:dnls1.2}
We need to derive  the relevant  linearized equations. Our main interest is in the spectral stability of the periodic waves \eqref{dnl:sol}, constructed in Proposition \ref{prop:21}. If one tries directly the linearization ansatz suggested by the periodic  wave solution \eqref{dnl:sol}, the resulting linear equations are not in a very convenient
  form that could easily be analyzed, see the discussion after formula \eqref{70}. So,  it is better  to argue in the framework suggested by the  waves for the gauged equation \eqref{10}.  We shall need an orthogonality relation, see \eqref{ans:18}, which will make this possible.
In accordance with the formula \eqref{dnl:sol}, we take the ansatz
\begin{equation}
\label{ans:20}
q=(\phi(x-ct)+e^{\la t} \eta(x-c t)) e^{i (\om t +\f{c}{2}(x-ct)- \f{3}{4}\int_{-T}^{x-ct} \phi^2(y) dy)}.
\end{equation}
We can now write the eigenvalue problem for $\eta$ as follows - denoting $\si(x)=\f{c}{2} x - \f{3}{4}\int_{-T}^{x} \phi^2 dy$ and plug it in \eqref{dnl:10}. Ignoring $O(\eta^2)$ terms,  we obtain
\begin{eqnarray*}
0 &=& i(\la \eta -c \eta_y+i \eta(\om-c \si')) + (\eta_{yy}+2i \eta_y \si'+\eta(i\si''-(\si')^2))+ \\
&+& i [\phi^{2}\eta_y+2\phi^2 \Re \eta'+z(2\phi\phi'+i \si'\phi^2)+\Re \eta(4\phi \phi'+2\phi^2i \si')].
\end{eqnarray*}
Taking into account
$$
\si'=\f{c}{2}-\f{3}{4} \phi^2, \si''=-\f{3}{2} \phi\phi',
$$
and splitting $\eta=(\Re \eta, \Im \eta)$, we arrive at
 \begin{equation}
 \label{169a}
 \cj \cl \left(\begin{array}{c}
 \Re \eta \\ \Im \eta
 \end{array}
 \right)= \la \left(\begin{array}{c}
 \Re \eta \\ \Im \eta
 \end{array}
 \right),
 \end{equation}
 where
 \begin{eqnarray*}
 	\cj &=&\left(\begin{array}{cc}
 		0 & 1 \\ -1 & 0
 	\end{array}
 	\right),
 	\cl=\left(\begin{array}{cc}
 		\cl_1 & N \\ N^* & \cl_2
 	\end{array}
 	\right),\\
 	\cl_1 &=&  -\p_{yy}+\left(\om-\f{c^2}{4}\right)+\f{ 3 c}{2}\phi^{2}-\f{11}{16} \phi^{4} \\
 	N &=& -\f{3}{2} \phi^{2} \p_y+\f{3}{2} \phi \phi', \ \ N^* =  \f{3}{2} \phi^{2} \p_y+\f{9}{2} \phi \phi', \\
 	\cl_2 &=&  -\p_{yy}+\left(\om-\f{c^2}{4}\right)+\f{c}{2}\phi^{2}-\f{3}{16} \phi^{4}
 \end{eqnarray*}
By direct inspection,  $\cl \left(\begin{array}{c}
0  \\ \phi
\end{array}
\right)= 0$, whence
$$
0=\dpr{\cl \left(\begin{array}{c}
	\Re \eta \\ \Im \eta
	\end{array}
	\right)}{\left(\begin{array}{c}
	0  \\ \phi
	\end{array}
	\right)}=\la \dpr{\cj^{-1} \left(\begin{array}{c}
	\Re \eta \\ \Im \eta
	\end{array}
	\right)}{\left(\begin{array}{c}
	0  \\ \phi
	\end{array}
	\right)}=\la \dpr{\Re \eta}{\phi}.
$$
It follows that for $\la\neq 0$, $\dpr{\Re \eta}{\phi}=0$. Therefore, at least as far as the analysis of the linearized  problem is concerned, we might take $\eta:\Re\eta\perp \phi$. That is, it suffices to consider  the following simplified reduced  linearization
\begin{equation}
\label{ans:18}
q=(\phi(x-ct)+\eta(t, x-c t)) e^{i (\om t +\theta(x-ct)- \f{1}{2}\int_{-T}^{x-ct} \phi^2(y) dy)}, \Re \eta\perp \phi.
\end{equation}
 We would like to translate this particular form of the perturbation of $q$ into the corresponding problem for $u$ in the gauged equation. According to the gauge transformation \eqref{gauge}, we have that
$$
u(t,x)=q(t,x) e^{\f{i}{2} \int_{-T}^x |q(t,y)|^2 dy}
$$
Using the particular form of \eqref{ans:20} and expanding in powers of $\eta$, we obtain
\begin{eqnarray*}
	& & u(t,x) = (\phi(x-ct)+e^{\la t} \eta(x-c t))e^{i \om t} e^{i\left(\f{c}{2}(x-ct) - \f{1}{4}\int_{-T}^{x-ct} \phi^2(y) dy+\int_{-T}^{x-ct} \phi(y) \Re \eta(y)  dy\right)}+O(\eta^2)\\
	&=& \left(\phi(x-ct)+e^{\la t} \eta(x-c t)+i \phi(x-ct)\int_{-T}^{x-ct} \phi(y) \Re \eta(y)  dy \right)e^{i \om t}  e^{i\left(\f{c}{2}(x-ct) - \f{1}{4}\int_{-T}^{x-ct} \phi^2(y) dy\right)}+\\
	&+& O(\eta^2)=:(\phi(x-ct)+ e^{\la t} z(x-c t)) e^{i (\om t+\theta(x-ct))}+O(\eta^2),
\end{eqnarray*}
where we have made the assignment $z:=\eta +i \phi \int_{-T}^x \phi(y) \Re \eta(y)  dy$. Note that we still need to address the periodicity of $z$, and we do this now. We have, by the periodicity of $\eta$ and $\Re \eta\perp \phi$,
\begin{equation}
\label{ans:16}
z(T)=\eta(T)+i \phi \dpr{\phi}{\Re \eta}=\eta(T)=\eta(-T)=z(-T).
\end{equation}
and similarly, $z'(T)=z'(-T)$.
We have, for the  real and imaginary parts,
\begin{equation}
\label{ans:30}
\left\{\begin{array}{l}
\Re z= \Re \eta \\
\Im z= \Im \eta + \phi(x) \int_{-T}^x \phi(y) \Re \eta(y)  dy
\end{array}
\right.
\end{equation}
or conversely
\begin{equation}
\label{ans:35}
\left\{\begin{array}{l}
\Re \eta= \Re z  \\
\Im \eta= \Im z -  \phi(x) \int_{-T}^x \phi(y) \Re z (y)  dy
\end{array}
\right.
\end{equation}
In other words, starting with the appropriate form \eqref{ans:18}, we have represented
\begin{equation}
\label{ans:10}
u(t,x)=(\phi(x-ct)+e^{\la t} z(x-ct))e^{i (\om  t+\theta(x-ct))}+O(z^2)
\end{equation}
where $z$ is a periodic increment and incidentally, by virtue of \eqref{ans:35}, we also have $\Re z\perp \phi$. 

Plug in the formula \eqref{ans:10} in  \eqref{10}, and  using that $\phi$ satisfies \eqref{30}, and   ignoring all terms in the form $O(z^2)$, we arrive at  the following linear equation for the increment $z$
$$
i(\la z -c z_y+i z(\om-c \theta')) + (z_{yy}+2i z_y \theta'+z(i\theta''-(\theta')^2))+ i [\phi^{2}(z_y+i z \theta')+2\phi(\phi'+i\theta' \phi )\Re z]=0
$$
Noting
$$
\theta'=\f{c}{2}-\f{\phi^{2}}{4}, \ \  \theta''=-\f{1}{2} \phi\phi',
$$
we can write
\begin{eqnarray*}
	& &  i \la z +z_{yy} +\left[-\om+\f{c^2}{4}-i \f{1}{2} \phi \phi'-\f{c}{2} \phi^{2}  +\f{3}{(16} \phi^{4}                \right]z+\\
	&+&  i z_y \f{1}{2} \phi^{2}+2i  \phi \left(\phi' +i\f{c}{2}\phi-i\f{1}{4}\phi^{3}\right)\Re z=0
\end{eqnarray*}
Split $z$  in real and imaginary parts, namely $z=v+i w$. In terms of $v,w$, we have a linear system that reads as follows
\begin{eqnarray*}
	& & \la v+w_{yy}+   \f{1}{2} \phi^{2}v_y    +  \f{3}{2}\phi \phi'   v + \left[-\om+\f{c^2}{4}-\f{c}{2} \phi^{2} +\f{3}{(16}  \phi^{4}\right] w=0 \\
	& &
	- \la w+v_{yy} - \f{1}{2} \phi^{2}w_y+\left[-\om+\f{c^2}{4}-\f{3c}{2} \phi^{2}  +\f{11}{16} \phi^{4} \right]v+\f{1}{2}\phi \phi' w=0
\end{eqnarray*}
We can write the eigenvalue problem   in the form
\begin{equation}
\label{169}
\cj \cl \left(\begin{array}{c}
v \\ w
\end{array}
\right)=\la \left(\begin{array}{c}
v \\ w
\end{array}
\right)
\end{equation}
where
\begin{eqnarray*}
	\cj &=&\left(\begin{array}{cc}
		0 & 1 \\ -1 & 0
	\end{array}
	\right),
	\cl=\left(\begin{array}{cc}
		\cl_1 & M \\ M^* & \cl_2
	\end{array}
	\right), \\
	M &=& \f{1}{2} \phi^{2} \p_y-\f{1}{2} \phi \phi', \ \ M^* =  -\f{1}{2} \phi^{2} \p_y-\f{3}{2} \phi \phi'.
\end{eqnarray*}
Here, we  introduce the linearized operators associated with the  profile equation \eqref{30}, namely   the second order Schr\"odinger operators
\begin{eqnarray*}
	\cl_+ &=&  -\p_{yy} +(\om-\f{c^2}{4}) +\f{3c}{2}  \phi^{2}  -\f{15}{16} \phi^{4}            \\
	\cl_- =\cl_2 &=& -\p_{yy} +(\om-\f{c^2}{4}) +\f{c}{2}\phi^{2}-\f{3}{16}\phi^{4}.
\end{eqnarray*}
They would be instrumental in the eigenvalue problem, associated with the periodic waves   under consideration.
Record the  eigenvalue problem \eqref{169} in the compact form
\begin{equation}
\label{170}
 \cj \cl \vec{{\mathbb U}}=\la \vec{{\mathbb U}}, \vec{{\mathbb U}}\in H^2_{per}[-T,T]\times H^2_{per}[-T,T].
\end{equation}
 We note that the linearized  problem \eqref{170} is in the standard Hamiltonian  form $\cj \cl, \cj^*=-\cj$, $\cl^*=\cl$. As we have mentioned above the corresponding linearized problem one obtains for $\eta$ in \eqref{ans:18} is  not  as convenient, see the discussion after \eqref{70} below. However, our analysis shows that as first order approximations, this problem is equivalent to \eqref{170}, through the change of variables \eqref{ans:30}.  More precisely,  starting with $z$, which solves \eqref{170} for any $\la\neq 0$,
	we see that $\cl z=\la \cj^{-1} z$, whence $\cj^{-1} z\perp Ker(\cl)$. It follows that $\Re z=v\perp \phi$, as $\left(\begin{array}{c} 0 \\ \phi \end{array}\right)\in Ker(\cl)$. This can be seen directly, but see also  Proposition \ref{le:10} below.
	 Next, one substitutes $\eta$ instead of $z$ according to \eqref{ans:30} in \eqref{170}, we obtain the linearized problem for $\eta$.
Based on this analysis, we say that the waves given in \eqref{dnl:sol} are stable, if the eigenvalue problem \eqref{170} is stable. More formally,
\begin{definition}
	\label{defi:20}
	We say that the wave $\Phi(x-c t) e^{i \om t}e^{-\f{3 i}{4} \int_{-T}^{x-ct} \phi^2(y) dy+const.}$ is spectrally stable solution of \eqref{dnl:10}, if the eigenvalue problem \eqref{170} does not have non-trivial solution $(\la, \vec{{\mathbb U}})$, with $\Re \la>0$.
\end{definition}
{\bf Remark:} In principle, an instability for the waves would mean that there exists $\la:\Re\la>0$, so that $\la\in \si(\cj\cl)$. As all the potentials in $\cl$ are periodic, it is a standard fact that all possible solutions of \eqref{170} represent eigenvalues only\footnote{ Indeed, as the resolvent operators $(\cj\cl-\la)^{-1}, \la\in\rone, \la>>1$ are smoothing of order two, this guarantees that $(\cj\cl-\la)^{-1}:L^2[-T,T]\times L^2[-T,T] \to L^2[-T,T]\times L^2[-T,T] $ is compact, whence its spectrum consists entirely of eigenvalues converging towards zero. It follows that $\si(\cj\cl)$ consists of eigenvalues only.}.

\subsection{Main results: DNLS}
Before we  proceed with the statements, we shall need to introduce an object that will play a role in the statement.
First, it is established,  see Proposition \ref{le:10} below, that
$$
Ker(\cl)=span\left\{\left(\begin{array}{c} 0 \\ \phi \end{array}\right),   \left(\begin{array}{c} \phi'\\ -\f{\phi^3}{4} \end{array}\right) \right\},
$$
whence  $\cl:Ker(\cl)^\perp\to Ker(\cl)^\perp$ is well-defined unbounded operator.
Given that, we consider a symmetric $2\times 2$ matrix $D$, with entries
\begin{eqnarray*}
	D_{11}  &=&
	\dpr{\cl^{-1} \left(\begin{array}{c} \phi \\ 0 \end{array}\right)}{\left(\begin{array}{c} \phi \\ 0 \end{array}\right)} , \\
	D_{1 2}=D_{2 1}  &=&  \dpr{\cl^{-1} \left(\begin{array}{c} \phi \\ 0 \end{array}\right)}{
		\left(\begin{array}{c} \f{\phi^{3}}{4} \\ \phi'\end{array}\right)}, \\
	D_{2 2} &=&
	\dpr{\cl^{-1} \left(\begin{array}{c}  \f{\phi^{3}}{4} \\ \phi'\end{array}\right) }{
		\left(\begin{array}{c} \f{\phi^{3}}{4} \\ \phi'\end{array}\right)},
\end{eqnarray*}
Note that since  $\left(\begin{array}{c} \phi \\ 0 \end{array}\right), \left(\begin{array}{c}  \f{\phi^{3}}{4} \\ \phi'\end{array}\right) \in Ker(\cl)^\perp$,  the elements $\cl^{-1} \left(\begin{array}{c} \phi \\ 0 \end{array}\right),	\cl^{-1} \left(\begin{array}{c}  \f{\phi^{3}}{4} \\ \phi'\end{array}\right) \in Ker(\cl)^\perp$ are uniquely defined.
 \begin{theorem}
	\label{theo:10}
	Consider   the waves constructed in Proposition \ref{prop:21}, subject to the condition \eqref{32}. These waves represent all non-vanishing bell-shaped periodic  waves for \eqref{10}.
	
	These waves are spectrally stable if and only if the matrix $D$ defined above has exactly one negative eigenvalue. Equivalently (since from general considerations, $D$ has at most one negative eigenvalue), the waves are stable, if and only if
	\begin{equation}
	\label{crit}
		\det(D)=	\dpr{\cl^{-1} \left(\begin{array}{c} \phi \\ 0 \end{array}\right)}{\left(\begin{array}{c} \phi \\ 0 \end{array}\right)} 	\dpr{\cl^{-1} \left(\begin{array}{c}  \f{\phi^{3}}{4} \\ \phi'\end{array}\right) }{
			\left(\begin{array}{c} \f{\phi^{3}}{4} \\ \phi'\end{array}\right)} - \dpr{\cl^{-1} \left(\begin{array}{c} \phi \\ 0 \end{array}\right)}{
			\left(\begin{array}{c} \f{\phi^{3}}{4} \\ \phi'\end{array}\right)}^2<0.
	\end{equation}
\end{theorem}
We have the following corollary
\begin{corollary}
	\label{cor:pel1}
	The non-vanishing bell-shaped periodic waves considered in Theorem \ref{theo:10} are stable, provided $	\dpr{\cl^{-1} \left(\begin{array}{c} \phi \\ 0 \end{array}\right)}{\left(\begin{array}{c} \phi \\ 0 \end{array}\right)}<0$.
	Since, we also establish that
	 $$
	 sgn(	\dpr{\cl^{-1} \left(\begin{array}{c} \phi \\ 0 \end{array}\right)}{\left(\begin{array}{c} \phi \\ 0 \end{array}\right)} )=sgn(\dpr{\cl_+^{-1} \phi}{\phi}),
	 $$
	  an alternative stability criteria is  $\dpr{\cl_+^{-1} \phi}{\phi}<0$.
\end{corollary}
{\bf Remark:}   We have an explicit but long formula for $\dpr{\cl_+^{-1} \phi}{\phi}$, depending on the alternative parametrization of the waves in $g,\ka, \mu=\om-\f{c^2}{4}$ as in Proposition \ref{prop:10}. For example,  we have produced several representative slices of the graphs of $\dpr{\cl_+^{-1} \phi}{\phi}$ - see Figure \ref{fig:pic4} for $\mu=1$ and Figure \ref{fig:pic5} for $\mu=-1$.
	
	Related to this discussion, we observe from the  graphs that $\dpr{\cl_+^{-1} \phi}{\phi}$   changes sign over the three dimensional domain of parameters. One might somehow conjecture that in line with the Vakhitov-Kolokolov theory,   the condition
	$\dpr{\cl_+^{-1} \phi}{\phi}>0$ by itself, might  imply instability.  We have a result, which shows that this is not the case.
	\begin{corollary}
		\label{cor:pel2}
		There are  waves of the type described in Theorem \ref{theo:10}, which are stable and at the same time $	\dpr{\cl^{-1} \left(\begin{array}{c} \phi \\ 0 \end{array}\right)}{\left(\begin{array}{c} \phi \\ 0 \end{array}\right)}>0$.
	\end{corollary}
	For the proof, we mention first that we establish, see Section \ref{sec:5.2} below, that the quantities $D_{11}$ and $D_{12}$ do not vanish simultaneously. Thus, consider a parameter point $P_0$ for which  $D_{11}(P_0)=\dpr{\cl^{-1} \left(\begin{array}{c} \phi \\ 0 \end{array}\right)}{\left(\begin{array}{c} \phi \\ 0 \end{array}\right)}=0$ and $P_0\in \p\{P: D_{11}(P)<0\}$. As $D_{11}(P_0)=0$, it must be that  $D_{12}(P_0)\neq 0$. So,
	$$
	det(D(P_0))=D_{11}(P_0) D_{22}(P_0)-D_{12}^2(P_0)=-D_{12}^2(P_0)<0,
	$$
	whence this particular periodic wave is still stable. In fact, by the continuity with respect to parameters, $det(D)<0$ in a neighborhood, so all of these waves are still stable. On the other hand, as $P_0$ is on the boundary of $\{P:D_{11}(P)<0\}$,  for some of them $D_{11}(P)>0$. This completes the proof.
\subsection{Main results: quintic NLS}
Starting with the quintic NLS, \eqref{n:10}, after a change of variables
$u(t,x)\to \al u(b t, \sqrt{b} x)$ and $\al^4=\f{3}{16}$, we rescale to the following problem
\begin{equation}
\label{n:12}
i u_t+u_{xx}+\f{3}{16} |u|^4 u=0, -T\leq x\leq T,
\end{equation}
with a rescaled $T$, in comparison to \eqref{n:10}. This transformation allows us to consider \eqref{n:12}, instead of the more general \eqref{n:10}.

Evidently, plugging in the standing wave ansatz $u=e^{i \om t} \phi, \phi>0, \om>0$, we obtain the profile equation for the wave
\begin{equation}
\label{p:10}
-\phi''+\om \phi - \f{3}{16} \phi^5=0, -T\leq x\leq T.
\end{equation}
Clearly, this is exactly the profile equation \eqref{30}, with $c=0$. Consequently, we have the bell-shaped solutions described in Proposition \ref{prop:10}. We state the existence result.
\begin{proposition}
	\label{prop:12}
Let  $\om> 0$. Then, these are all  bell-shaped solutions of \eqref{p:10}:
	\begin{enumerate}
		\item If
		$$
		-\frac{64\sqrt{12}}{9}\om^{\f{3}{2}} <a<0,
		$$
		then, $R(\vp)=a$ has three roots $\vp_1<0<\vp_2<\vp_3$ and $\vp(0)=\vp_3,  \vp(-T)=\vp(T)=\vp_2$, described by
		\begin{equation}
		\label{66}
		\left\{ \begin{array}{ll}
		\varphi_1+\varphi_2+\varphi_3=0, \\
		\\
		\varphi_1\varphi_2+\varphi_1\varphi_3+\varphi_2\varphi_3=-16\om, \\
		\\
		\varphi_1\varphi_2\varphi_3=a.
		\end{array} \right.
		\end{equation}
		The solution is then given by
		\begin{eqnarray}
		\label{2.52}
		\phi^2(\xi)=\varphi(\xi) &=& \frac{\varphi_3(\varphi_2-\varphi_1)+\varphi_1(\varphi_3-\varphi_2)
			sn^2\left(
			\frac{\xi}{2g}, \kappa\right)}{(\varphi_2-\varphi_1)+(\varphi_3-\varphi_2)sn^2\left(
			\frac{\xi}{2g}, \kappa\right)}, \\
		\label{r:26}
		g &=&\frac{2}{\sqrt{\varphi_3(\varphi_2-\varphi_1)}}, \ \
		\kappa^2=-\frac{\varphi_1(\varphi_3-\varphi_2)}{\varphi_3(\varphi_2-\varphi_1)}\in (0,1).
		\end{eqnarray}
		
		\item  For every $a>0$, there is unique solution $\vp_3$ of $a=R(\vp)$, with $\vp_3>0$. As a consequence, there is unique bell-shaped solution $\vp: \vp(0)=\vp_3$, $\vp(-T)=\vp(T)=0$.
	\end{enumerate}
\end{proposition}
Thus, we are interested in the solutions described in \eqref{2.52}. Linearizing around the traveling wave
$e^{i \om t} \phi$, $u= e^{i \om t} (\phi+v)$, yields the eigenvalue  problem for
$\vec{{\mathbf v}}=e^{\la t} (\Re v, \Im v)$,
\begin{equation}
\label{l:14}
\left(\begin{array}{cc}
0 & 1 \\
-1 & 0
\end{array}\right)\left(\begin{array}{cc}
\cl_+ & 0 \\
0 & \cl_-
\end{array}\right)\vec{{\mathbf v}}=\la \vec{{\mathbf v}}.
\end{equation}
\begin{theorem}
	\label{theo:20}
Let $\om>0$ and $\phi$ are the bell-shaped traveling wave of the quintic NLS,
described in \eqref{2.52}, which alternatively can be parametrized by \eqref{r:26} and
$$
g\in (0, \infty), \ \ka\in (0,1), \ \om=\f{\sqrt{1-k^2+k^4}}{4 g^2}.
$$
Then, these solutions are stable, whenever $\dpr{\cl_+^{-1} \phi}{\phi}<0$.
\end{theorem}

The plan of the paper is as follows - the main object of investigation, namely the DNLS problem is considered in all sections, but the last one. More specifically,   in Section \ref{sec:2}, we introduce the basics of the instability index counting theory. In Section \ref{sec:2.2}, we construct small waves via a variational method. This is, on one hand standard, but we find it  useful in the sequel, as it provides an important piece of spectral information\footnote{which proved to be extremely non-trivial to obtain with the explicit waves under consideration}, namely that the scalar linearized operator $\cl_+$ has a single negative eigenvalue, for all points in the  parameter space. This property is then established for all linearized operators (about the waves of interest) via a topological arguments, as eigenvalues of $\cl_+$ are shown  not to cross the zero eigenvalue, see Proposition \ref{prop:23} later on. In Section \ref{sec:3}, we study the spectral properties of $\cl_\pm$ - first we need and present an alternative parametrization of the waves, see Section \ref{sec:3.1}, and then we describe the first few elements of $\si(\cl_\pm)$, see Proposition \ref{prop:23}. In Section \ref{sec:4}, we use the spectral  information from Section \ref{sec:3} to study the properties of the matrix Hill   operator $\cl$, which arises in the linearized problem. Namely, we show that its kernel is always two dimensional in Proposition \ref{le:10}. Note that this is the minimal dimension dictated by the N\"other's theorem, as the Hamiltonian system has two symmetries.    It is at this point that we start introducing some concrete calculations,
based on the formulas for the waves\footnote{Interestingly, we need to
	resort to differentiation with respect to parameters. This is always tricky, as the period generally depends on these parameters and one needs to appropriately prepare the problem by
	rescaling  to a fixed period, see Section \ref{sec:3.2} and Section \ref{sec:5.1} for specifics about these calculations}, see Section \ref{sec:4.2}.   In Section \ref{sec:5}, we wrap up the proof of the stability criteria for DNLS waves.

 In Section \ref{sec:6}, we study the stability of the quintic NLS waves. These turn out to be a two parameter subfamily of the three parameter family of DNLS waves considered earlier. One can compute the quantity $\dpr{\cl_+^{-1} \phi}{\phi}$, but in this case, the index counting theory stipulates  that the spectral stability is exactly equivalent to $\dpr{\cl_+^{-1} \phi}{\phi}<0$. We have an explicit, but long formula, which shows the intervals of stability for each given point in the parameter space - some graphs are given at the end of Section \ref{sec:6}, which illustrate where this is the case.

\section{Some preliminaries}
\label{sec:2}
{First, we introduce some notations. Let $S$ be a self-adjoint operator,  with domain $D(S)\subset L^2$, which is bounded from below, i.e. $\inf_{\|u\in D(S): \|u\|_{L^2}=1} \dpr{S u}{u}>-\infty$. Very often, such operator have only  (real) eigenvalues in their spectrum, each with finite multiplicity. For example, this is the case when $(S-\lambda I)^{-1}$ is a compact operator for some $\la: \la>>1$, which would be the main situation considered herein.  In such case, we denote their real eigenvalues $\la_0(S)\leq \la_1(S)<\ldots$. In particular, from the min-max characterizations, we have that
	$$
	\la_0(S)=\inf_{u\in D(S): \|u\|_{L^2}=1} \dpr{S u}{u}, \ \ \la_1(S)=\sup_{\zeta\neq 0} \inf_{u\in D(S)\cap \{\zeta\}^\perp: \|u\|_{L^2}=1} \dpr{S u}{u}
	$$
Next, we present some classical results about the instability index count theories. These allow us to count the number of unstable eigenvalues for eigenvalue problems of the form \eqref{170}, based  on the information about the self-adjoint portion $\cl$ and some specific quantities, which are also, in principle,  computable.
\subsection{Instability index theory}
\label{2.1}
We use the instability index count theory, as developed in \cite{Pel, KKS,KKS2, Kap}, see also \cite{LZ}. We present a corollary, which is enough  for  our purposes. For eigenvalue problem in the form
\begin{equation}
\label{1701}
\ci \ch \vec{{\mathbb U}}=\la \vec{{\mathbb U}}.
\end{equation}
we assume that $\ch=\ch^*$ has $dim(Ker(\ch)<\infty$, and also a finite number of negative eigenvalues, $n(\ch)$, a quantity sometimes referred to as Morse index of the operator $\ch$.   In addition, $\ci^*=-\ci$ and we shall require that $\ci^{-1}: Ker[\ch]\to Ker[\ch]^\perp$.
Let $k_r$  be  the number of positive eigenvalues of the spectral problem \eqref{1701} (i.e. the number of real instabilities or real modes), $k_c$ be the number of quartets of eigenvalues with non-zero real and imaginary parts, and $k_i^-$, the number of pairs of purely imaginary eigenvalues with negative Krein signature.  For a simple pair of imaginary eigenvalues $\pm i \mu, \mu\neq 0$, and the corresponding eigenvector
$\vec{z} = \left(\begin{array}{c}
z_1  \\ z_2
\end{array}\right) $, the Krein signature is
$
sgn(\dpr{ \ch  \vec{z}}{ \vec{z}}),
$
see \cite{KKS}, p. 267.

The matrix $D$ is introduced as follows  - for $Ker[\ch]=span \{\zeta_1, \ldots, \zeta_n\}$
\begin{equation}
\label{dij}
D_{i j}:=\dpr{\ch^{-1} [\ci^{-1} \zeta_i] }{\ci^{-1} \zeta_j}.
\end{equation}
Note that the last formula makes sense, since $ \cj^{-1} \zeta_i \in Ker[\ch]^\perp$. Thus $\ch^{-1}[\cj^{-1} \zeta_i]\in Ker[\ch]^\perp$ is well-defined. The index counting theorem, see Theorem 1, \cite{KKS2} states that if $det(D)\neq 0$, then
\begin{equation}
\label{e:20}
k_{Ham}:=k_r+2 k_c+2 k_i^-= n(\ch)-n(D).
\end{equation}
Note that $k_{Ham}=0$ guarantees spectral  stability\footnote{But note that this is not necessary. For example,  one might have $k_{Ham}=2=2k_i^-, k_r=k_i^-=0$, which means that no instabilities are present, but there is a pair of purely imaginary eigenvalues, with a negative Krein signature.}.   A particularly useful corollary of this result occurs when $n(\ch)=1$, since then the stability is equivalent to $k_{Ham}=0$, see \eqref{e:20}. Clearly,  the stability is equivalent to $n(D)=1$, whereas $n(D)=0$ leads to $k_{Ham}=1=k_r$, hence instability.

\subsection{Variational construction of small waves}
\label{sec:2.2}
This section constructs variational solution for the profile equation \eqref{30}. This may seem redundant, given the fact that we are able to construct, in a fairly explicit manner (i.e. with explicit dependence on the parameters), all solutions of interest to it. This is all so, but the variational construction yields an important additional property of these solutions that arise as constrained minimizer, which will be relevant later on. Namely, they will have the important property that $n(\cl_+)=1$, which turns out hard to verify in this context.
\begin{proposition}
	\label{prop:43}
	Let $T>0$, $c\in \rone$. Then, there exists $\eps_0=\eps_0(T,c) >0$, so that the variational problem
	\begin{equation}
	\label{v:10}
	\left\{\begin{array}{l}
J[v]=	\f{1}{2} \int_{-T}^T |v'(x)|^2 +\f{c}{8} \int_{-T}^T |v(x)|^{4} dx - \f{1}{32} \int_{-T}^T |v(x)|^{6} dx \to min \\
	\int_{-T}^T |v(x)|^2 dx=\eps
	\end{array}
	\right.
	\end{equation}
	has solution $V$ for every $0<\eps<\eps_0$. Moreover, it is a bell-shaped function, which satisfies the Euler-Lagrange equation
	\begin{equation}
	\label{EL:10}
	-V''+\zeta V+ \f{c}{2}V^{3}-\f{3}{16} V^{5} =0, -T<x<T,
	\end{equation}
	where $\zeta=\zeta(c,\eps, T)$ is the Euler-Lagrange multiplier.
	In addition, the linearized operator
	$$
	L_+:=-\p_{xx}+\zeta + \f{3 c}{2} V^{2} - \f{15}{16} V^{4}
	$$
	 has exactly one negative eigenvalue.
\end{proposition}
 \begin{proof}
 	We first need to check that the variational problem \eqref{v:10} is well-posed. That is, for sufficiently small $\eps$ and under the constraint $\|v\|_{L^2}^2=\eps$, the functional $J$ is bounded from below. Indeed, from  the Gagliardo-Nirenberg-Sobolev (GNS)  inequality, we have
 	$$
 	\|v\|_{L^6[-T,T]}^6\leq C \|v\|_{\dot{H}^{\f{1}{3}}[-T,T]}^6\leq
 	C \|v'\|_{L^2}^2 \|v\|_{L^2}^4=C\eps^2 \|v'\|_{L^2}^2.
 	$$
 	Similarly, (with $c\neq 0$ as in the definition of $J(v)$, if $c=0$, just skip this step)
 	$$
 	\|v\|_{L^4[-T,T]}^4\leq
 	C \|v'\|_{L^2} \|v\|_{L^2}^3=C \|v'\|_{L^2} \eps^{\f{3}{2}} \leq \f{1}{4c} \|v'\|_{L^2}^2 + C \eps^3
 	$$
 	This allows one to estimate $J$ from below
 	\begin{equation}
 	\label{v:20}
 	J[v]\geq \left(\f{1}{4}-C\eps^2\right) \|v'\|^2 - C\eps^3>- C\eps^3
 	\end{equation}
 	provided $\eps: C\eps^2 \leq \f{1}{4}$. This shows the well-posedness.
 	
 	 Recall the Szeg\"o inequality $\|v_x\|_{L^2[-T,T]}\geq \|v^*_x\|_{L^2[-T,T]}$, where $v^*$ is the decreasing rearrangement of  $v$ as previously defined. Note that  the equality holds only when $v=v^*$, i.e. if $v$ is bell-shaped.  At the same time, for all $1\leq p\leq \infty$, there is $\|v\|_{L^p[-T,T]}=\|v^*\|_{L^p[-T,T]}$. This shows that $J[v]\geq J[v^*]$, while $\|v^*\|_{L^2}^2=\|v\|_{L^2}^2=\eps$. Thus, it suffices to restrict
 	 the variational problem \eqref{v:10} to bell-shaped entries only.
 	
 	 We now show that \eqref{v:10} has solutions. Let $J_*=\inf_{\|v\|^2=\eps} J[v]$ and pick a minimizing sequence of bell-shaped functions, $v_n: \|v_n\|^2=\eps,  J[v_n]\to J_*$. It follows from \eqref{v:20} that
 	 $$
 	 \limsup_n \|v_n\|_{H^1[-T,T]}<\f{C\eps^3 + J_*}{\f{1}{4}-C\eps^2}<\infty.
 	 $$
 	 Thus, $\{v_n\}_n $ is a bounded sequence in $H^1[-T,T]$, hence a precompact in $L^2[-T,T]$. Thus, we may extract a subsequence, which converges weakly in $H^1$ and strongly in $L^2$.  Without loss of generality, the subsequence  is $v_n$, say $\lim_n \|v_n-V\|_{L^2}=0$. By the GNS inequality and the $\sup_n \|v_n\|_{H^1}<\infty$, it follows that $\lim_n \|v_n-V\|_{L^p[-T,T]}=0, 1<p<\infty$. In particular, $v_n\to V$ in
 	 $L^4, L^6$. At the same time, by the lower semi-continuity of the $H^1$ norm, with respect to weak convergence,
 	 $\liminf_n \|v_n\|_{H^1}\geq \|V\|_{H^1}$. It follows that
 	 $$
 	 J_*=\liminf_n J[v_n]\geq J[V].
 	 $$
 	while $\|V\|_{L^2}^2=\eps$. Thus, $V$ is a solution of \eqref{v:10} and it is a bell-shaped as a limit of bell-shaped functions. It now remains to establish the Euler-Lagrange equation \eqref{EL:10} and $n(L_+)=1$. This is all very standard. Fix $h: \dpr{h}{V}=0$ and consider
 	$$
 	f(\de):= J\left[\sqrt{\eps} \f{V+\de h}{\|V+\de h\|^2} \right]
 	$$
 	Since $V$ is a minimizer of \eqref{v:10}, it follows that $f$ has a minimum at $\de=0$. Thus, $f'(0)=0$, which yields exactly \eqref{EL:10}. Furthermore, $f''(0)\geq 0$, which amounts to
 	$$
 	\dpr{L_+ h}{h}\geq 0, h\perp V
 	$$
 	Thus, $n(L_+)\leq 1$. On the other hand, by direct inspection, $L_+[V']=0$ and the function $V'$ has two zeros, at zero and at $T$.  Thus, this is not the ground state, which needs to be positive, hence there is a negative eigenvalue, whence $n(L_+)=1$.
 \end{proof}

\section{Spectral properties of $\cl_\pm$}
\label{sec:3}
We   need to establish some useful spectral  properties   for the scalar Schr\"odinger operators $\cl_\pm$, such as \eqref{d:12}.  In addition, we shall also need to compute various quantities involving $\cl_+^{-1} \phi$. This requires explicit calculations involving the waves, so we start with an alternative parametrization, which will be useful in the actual computations.

\subsection{An alternative  parametrization of the waves}
\label{sec:3.1}
As we shall see, it is possible to obtain  formulas in terms of the roots $\vp_1, \vp_2, \vp_3$. These are not necessarily  good variables to work with.
We introduce a new set  of parameters. Namely, we shall use $g, \ka$ and
$
\mu=16(\om-\f{c^2}{4}).
$
Based on that and the types of solutions that we consider, it is convenient to further distinguish between the cases, $\mu>0$ and $\mu<0$.
\subsubsection{The case $\mu>0$}
In terms of this variables, (see \eqref{r:24} for the formulas connecting the roots to $g, \ka$), we can express the roots as follows
\begin{eqnarray}
\label{r:10}
\vp_1\vp_2 &=& \f{1}{3}\left(-\f{4}{g^2}+\f{8\ka^2}{g^2}-\mu\right)=:-A(g,\ka,\mu)\\
\label{r:20}
\vp_2\vp_3 &=& \f{1}{3}\left(\f{8}{g^2}-\f{4\ka^2}{g^2}-\mu\right)=:B(g,\ka,\mu)\\
\label{r:30}
\vp_1\vp_3 &=& \f{1}{3}\left(-\f{4}{g^2}-\f{4\ka^2}{g^2}-\mu\right)=:-C(g,\ka,\mu).
\end{eqnarray}
Recall that we are interested in a case, where $\vp_1, \vp_2, \vp_3$ are all real and $\vp_1<0<\vp_2<\vp_3$. The assumption $\om-\f{c^2}{4}>0$ ensures $\vp_3>0$.  It is easy to see that the rest  is equivalent to $\vp_1\vp_2<0$ and $\vp_2\vp_3>0$. Indeed, $\vp_1\vp_2<0$ rules out complex eigenvalues (since then $\vp_1\vp_2=\vp_1\bar{\vp}_1>0$). Additionally, $\vp_2\vp_3>0$ rules out the possibility $\vp_1<\vp_2<0<\vp_3$. Thus, in the case under consideration, namely $\mu>0$,
$$
\vp_1<0<\vp_2<\vp_3 \Longleftrightarrow -A=\vp_1\vp_2<0 \ \&\  B=\vp_2\vp_3>0.
$$
Working out these inequalities leads to the following  conditions on the new parameters
\begin{equation}
\label{r:32}
\mu>0\ \ \&\ \ 0<g^2< \f{8}{\mu}\ \ \&\ \ 0<\ka^2<\min\left(\f{4+\mu g^2}{8},
\f{8-\mu g^2}{4}\right).
\end{equation}
Note that $\min\left(\f{4+\mu g^2}{8},
\f{8-\mu g^2}{4}\right)\leq 1$, so the standard restriction for $\ka\in (0,1)$ is not violated. In fact, for the case $\mu>0$,
\begin{equation}
\label{r:36}
 \vp_1<0<\vp_2<\vp_3  \Longleftrightarrow   0<g< \sqrt{\f{8}{\mu}}\ \ \&\ \ 0<\ka^2<\min\left(\f{4+\mu g^2}{8},
\f{8-\mu g^2}{4}\right).
\end{equation}
For future reference, we need the formula for $\vp_1, \vp_2, \vp_3, c$ in terms of the new variables. We have from Viet's formulas and \eqref{r:10}, \eqref{r:20}, \eqref{r:30},
\begin{eqnarray}
\label{r:50}
\vp_1 &=&-\sqrt{\f{AC}{B}}, \vp_2=\sqrt{\f{AB}{C}}, \vp_3=\sqrt{\f{BC}{A}}, \\
\label{r:51}
c &=& \f{AB+BC-AC}{4 \sqrt{ABC}}
\end{eqnarray}

\subsubsection{The case $\mu\leq 0$}
This case, $\mu \leq 0$  is very similar to the case $\mu>0$ - all the formulas stay unchanged, while the regions of validity, such as \eqref{r:32} change. More specifically, \eqref{r:10}, \eqref{r:20}, \eqref{r:30} remain unchanged, but now, we have to find new constraints corresponding that $\vp_1, \vp_2, \vp_3$ are real and $\vp_1<0<\vp_2<\vp_3$. So, we need to
enforce $A>0, B>0, C>0$. Note that $B>0$ is automatic, due to the inequalities $\ka<1$ and $\mu\leq 0$. Also, note that since $A<C$, it is enough to enforce $A>0$. This gives rise to the new constraints, similar to \eqref{r:36}, namely - for $\mu<0$,
\begin{equation}
\label{r:361}
 \vp_1<0<\vp_2<\vp_3  \Longleftrightarrow   \ 0<g<\sqrt{-\f{4}{\mu}}  \ \&\ \ 0<\ka^2<\f{4+\mu g^2}{8}.
\end{equation}
The case $\mu=0$ can be naturally considered as part of \eqref{r:361}, so we get
$$
\vp_1<0<\vp_2<\vp_3  \Longleftrightarrow   \ 0<g<+\infty  \ \&\ \ 0<\ka^2<\f{1}{2}.
$$
Combining the results from the cases $\mu>0$, $\mu\leq 0$, we can formulate the new parametrization in the following proposition.
\begin{proposition}
	\label{prop:50}
	Let
	\begin{equation}
	\label{382}
	\mu\in \rone \ \ \& \ \ 0<g<\sqrt{\max\left(\f{8}{\mu}, -\f{4}{\mu}\right)}\ \ \& \ \
	\ka^2 \leq \min\left(\f{4+\mu g^2}{8},
	\f{8-\mu g^2}{4}\right).
	\end{equation}
	Then, the formulas \eqref{r:50} and \eqref{2.5}  describe
	all non-vanishing bell-shaped solutions, with
	$$
	\vp(0)=\vp_3>\vp(T)=\vp(-T)=\vp_2>0,
	$$
	constructed in Proposition \ref{prop:10}.
\end{proposition}
Now that we have the alternative description of the waves, it is time to establish some further structural facts about the first few eigenvalues in the spectrums of $\cl_\pm$.
\subsection{Description of the spectrum of $\cl_\pm$}
\label{sec:3.2}
\begin{proposition}
	\label{prop:23}
	For all bell-shaped  waves constructed in Proposition \ref{prop:10},
	the scalar linearized Schr\"odinger operators $\cl_\pm$, with $D(\cl_\pm)=H^2_{per}(-T,T)$ have the properties
	\begin{enumerate}
		\item $\cl_-\geq 0$, $\la_0(\cl_-)=0$, with $Ker[\cl_-]=span[\phi]$, $\la_1(\cl_-)>0$.
		\item $\cl_+$ has exactly one simple negative eigenvalue $\la_0(\cl_+)<0$ (say with a ground state $\chi_0$),  it has a simple eigenvalue at zero, $\la_1(\cl_+)=0$, with
		$Ker[\cl_+]=span[\phi']$, and $\la_2(\cl_+)>0$. In particular, there exists $\de>0$, so that
		\begin{equation}
		\label{86}
		\inf_{u\perp \chi_0, u\perp \phi'} \dpr{\cl_+ u}{u}\geq \de \|u\|_{L^2}^2.
		\end{equation}
	\end{enumerate}
\end{proposition}
\begin{proof}
	The statement for $\cl_-$ is straightforward. 	Indeed, by a direct check $\cl_-\phi=0$, whence $0$ is an eigenvalue.  Since $\phi$ does not change sign, it means that $0$ is a simple eigenvalue at the bottom of $\si(\cl_-)$. Hence, $\la_0(\cl_-)=0<\la_1(\cl_-)$ and $\cl_-|_{span[\phi]^\perp}\geq \de>0$.
	
	We now turn our attention to the spectral properties of $\cl_+$. One issue complicating matters is the dependence of the period on the variables $g, \ka$, which makes differentiation with respect to them problematic. In order to avoid this dependence, we introduce a scaling transformation. Namely, a new function $Q: \phi(\xi)=Q\left(\f{\xi}{T}\right)$ is introduced, which is $2$ periodic.
	Then, the new equation that we need to consider is
	\begin{equation}
	\label{d:10}
	-Q''+\f{T^2 \mu}{16} Q+\f{c T^2}{2} Q^3 - \f{3 T^2}{16} Q^5=0, -1<\eta<1.
	\end{equation}
	Then, the new linearized operator relevant to this problem is
	$$
	\tilde{\cl}_+:=-\p_{\eta\eta}+\f{T^2 \mu}{16}  + \f{3 cT^2}{2} Q^2 - \f{15 T^2}{16} Q^4,
	$$
	with $D(\tilde{\cl}_+)=H^2_{per.}[-1,1]$. One can also  see that $\tilde{\cl}_+[Q']=0$ by differentiating \eqref{d:10}.
	
	It is clear now that the results we want to establish are equivalent to
	\begin{equation}
	\label{d:12}
	\la_0(\tilde{\cl}_+)<\la_1(\tilde{\cl}_+)=0<\la_2(\tilde{\cl}_+); \ \ Ker[\tilde{\cl}_+]=span[Q'].
	\end{equation}
	So, our goal is to show \eqref{d:12}. Since, $\tilde{\cl}_+[Q']=0$, $Q'$, zero is an eigenvalue and $Q'$ is an eigenfunction.  It is also  clear that $\la_0(\tilde{\cl}_+)<0$, since $Q'$ is an eigenfunction at zero and it changes sign. Thus, one conclude that the ground state eigenvalue is negative.
	
	Our plan for the rest of the proof is as follows - we need to show that
	\begin{enumerate}
		\item $Ker[\tilde{\cl}_+]=span[Q']$ for all values of the parameters $(g,\ka,\mu)$ described in \eqref{382}.
		\item $n(\tilde{\cl}_+(\mu_0, g_0, \ka_0) =1$ for some value $(\mu_0, g_0, \ka_0)$ in the parameter space.
	\end{enumerate}
	We claim that this will be enough to establish \eqref{d:12}.  Note  first, that the set
	$$
	\ca:=\left\{\mu\in \rone \ \ \& \ \ 0<g<\sqrt{\max\left(\f{8}{\mu}, -\f{4}{\mu}\right)}\ \ \& \ \
	\ka^2 \leq \min\left(\f{4+\mu g^2}{8},
	\f{8-\mu g^2}{4}\right)\right\}.
	$$
	is an open and connected set in $\rthree$, and  the maps
	$(\mu, g, \ka)\to \la_j(\tilde{\cl}_+(\mu, g, \ka))$,
	$j=0,1,2, \ldots$ are continuous in $\mu, g, \ka$.  Since
	\begin{equation}
	\label{s:08}
	\la_0(\tilde{\cl}_+(\mu_0, g_0, \ka_0))< \la_1(\tilde{\cl}_+(\mu_0, g_0, \ka_0))=0< \la_2(\tilde{\cl}_+(\mu_0, g_0, \ka_0)),
	\end{equation}
	such inequality must persist for all $(\mu, g, \ka)\in \ca$. Indeed,  assume for a contradiction that for some other value $(\mu_1, g_1, \ka_1)\in \ca$,
	$$
	\la_0(\tilde{\cl}_+(\mu_1, g_1, \ka_1))<  \la_1(\tilde{\cl}_+(\mu_1, g_1, \ka_1)) <  \la_2(\tilde{\cl}_+(\mu_1, g_1, \ka_1))=0.
	$$
	Arguing by continuity, an eigenvalue crosses from being positive to being negative, implying that in some intermediate point, there  is a multiplicity two  eigenvalue at
	zero, which is a contradiction with $Ker[\tilde{\cl}_+]=span[Q']$ for all values of the parameters $(g,\ka,\mu)$.
	
	\subsubsection{Proof of $Ker[\tilde{\cl}_+]=span[Q']$}
	\label{sec:3.2.1}
	Given what we have established already,
	the only remaining fact that we need to establish is that there is no eigenfunction  $\psi\notin span[\chi_0, Q']$, corresponding to zero   eigenvalue. Assuming that such an eigenfunction  does exist (and without loss of generality orthogonal to $\chi_0, Q'$), we will reach a contradiction. First, by Sturm oscillation theory, $\psi$ should have two zeros in $[-1,1)$  Since it is orthogonal to $Q'$, the function $\psi$ must be even, with zeros at $\pm x_0: 0<x_0<1$. Without loss of generality $\psi(x)>0: x\in (-x_0, x_0)$, while $\psi(x)<0, x\in (-1, x_0)\cup (x_0, 1)$.
	
	Our approach is as follows. We construct elements in $Ker(\tilde{\cl}_+)^\perp$ and then we use them to contradict the existence of such $\psi$.  To that end,
	a relation that is immediately useful is
	\begin{equation}
	\label{w:5}
	\tilde{\cl}_+[Q]= -T^2[ -c Q^3 + \f{3}{4} Q^5].
	\end{equation}
	Another one is to take a derivative with respect to $\mu$ in \eqref{d:10}. Recall, see \eqref{61},  that $T=2g K(\ka)$, so it is independent on $\mu$. We obtain
	\begin{equation}
	\label{j:10}
	\tilde{\cl}_+[\p_\mu Q]=-T^2[\f{1}{16} Q+ \f{\p_\mu c}{2} Q^3].
	\end{equation}
	Finally, we take a derivative with respect to $\ka$. We get
	\begin{eqnarray}
	\label{j:16}
	\tilde{\cl}_+[\p_\ka Q] &=&  2T T_\ka[-\f{\mu}{16} Q-\f{c}{2} Q^3+\f{3}{16} Q^5]-T^2\f{c_\ka}{2} Q^3=\\
	\nonumber
	&=& -T^2\left[\f{\mu K'(\ka)}{8 K(\ka)} Q+(\f{c K'(\ka)}{K(\ka)}+\f{c_\ka}{2}) Q^3-\f{3 K'(\ka)}{8 K(\ka) }Q^5\right].
	\end{eqnarray}
	Formulas \eqref{w:5}, \eqref{j:10}, \eqref{j:16} allow us to solve for   $Q, Q^3$, provided
	\begin{equation}
	\label{j:21}
	c+c_\ka \f{K(\ka)}{K'(\ka)} - 2\mu c_\mu\neq 0.
	\end{equation}
	More precisely, isolating $Q,Q^3$, we obtain the system
	\begin{equation}
	\label{j:19}
	\left(\begin{array}{cc}
	\f{\mu}{4} & c+\f{c_\ka K(\ka)}{K'(\ka)} \\
	\f{1}{16} & \f{c_\mu}{2}
	\end{array}\right) \left(\begin{array}{c}
	Q \\ Q^3
	\end{array}\right) = -T^{-2} \left(\begin{array}{c}
	\tilde{\cl}_+(Q+ 2\f{K(\ka)}{K'(\ka)} Q_\ka \\ \tilde{\cl}_+(Q_\mu)
	\end{array}\right)
	\end{equation}
	We obtain
	\begin{equation}
	\label{j:22}
	\tilde{\cl}_+^{-1}(Q)= 16 T^{-2} \f{\f{c_\mu}{2} Q+c_\mu \f{K(\ka)}{K'(\ka)} Q_\ka-\left(c+c_\ka \f{K(\ka)}{K'(\ka)}\right)Q_\mu}{c+c_\ka \f{K(\ka)}{K'(\ka)} - 2\mu c_\mu}
	\end{equation}
	and there is a similar formula for $\tilde{\cl}_+^{-1}(Q^3)$, with the same denominator. Clearly, \eqref{j:21} is then a  solvability condition that ensures that $Q,Q^3\in Ran(\tilde{\cl}_+)\subset Ker[\tilde{\cl}_+]^\perp$. We have computed and plotted the function in \eqref{j:21}, see Figure \ref{fig:pic1}, Figure \ref{fig:pic2} that confirm the solvability condition \eqref{j:21}.
	\begin{figure}
		\centering
		\includegraphics[width=0.7\linewidth]{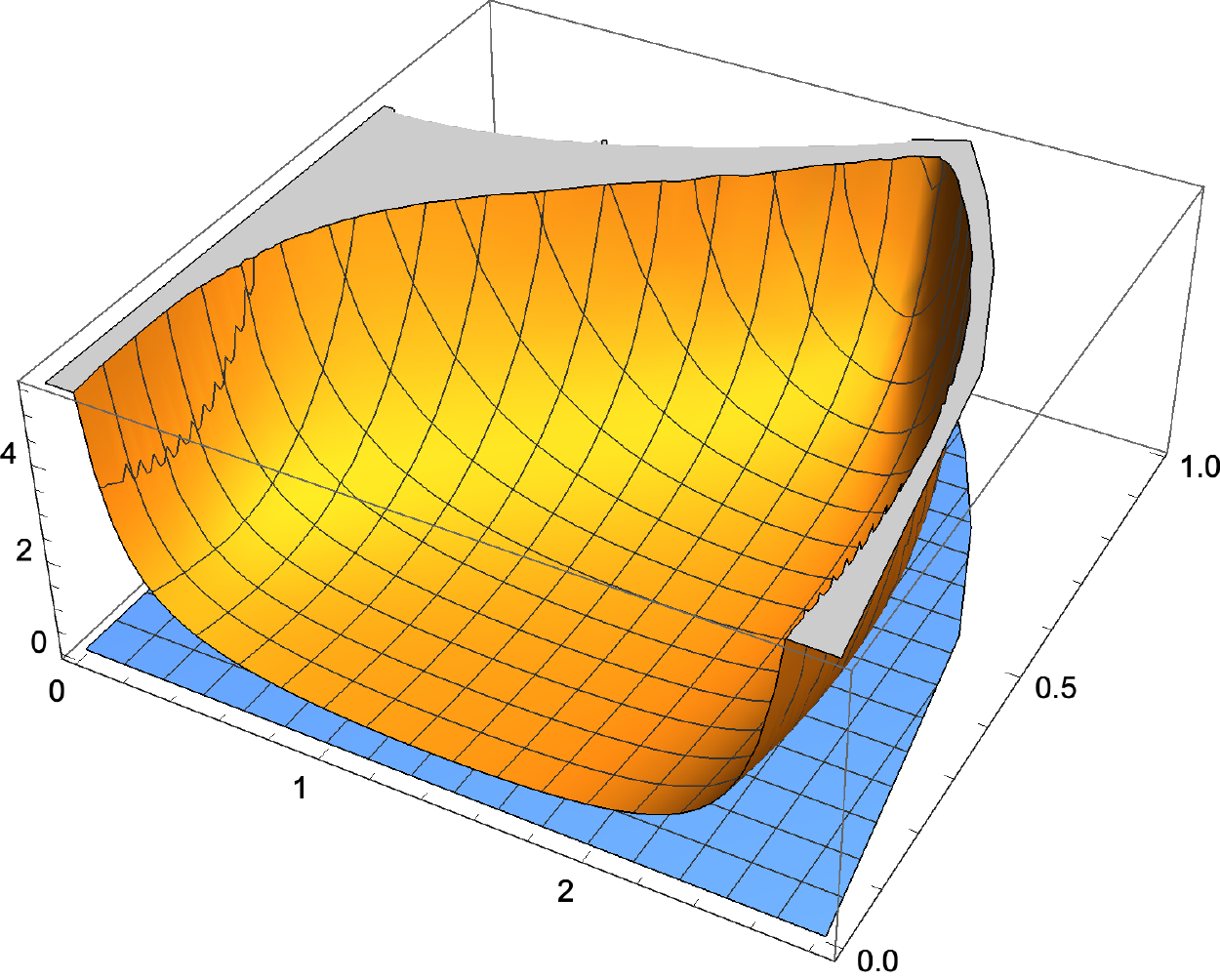}
		\caption{Graph of 	$c(g, \ka, \mu)+c_\ka(g, \ka, \mu) \f{K(\ka)}{K'(\ka)} - 2\mu c_\mu(g, \ka, \mu)$ , for $\mu=1$}
		\label{fig:pic1}
	\end{figure}
	\begin{figure}
		\centering
		\includegraphics[width=0.7\linewidth]{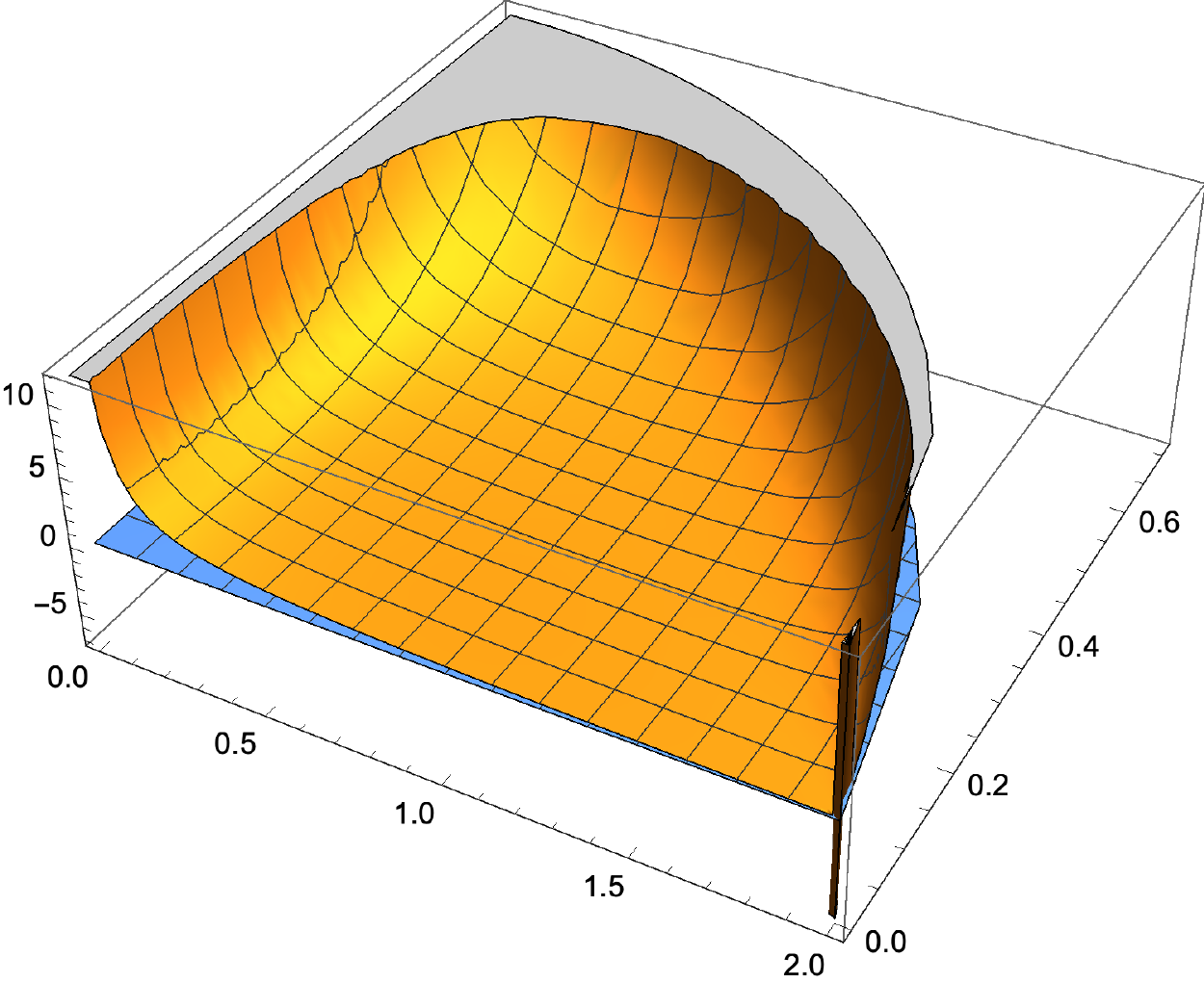}
		\caption{Graph of 	$c(g, \ka, \mu)+c_\ka(g, \ka, \mu) \f{K(\ka)}{K'(\ka)} - 2\mu c_\mu(g, \ka, \mu)$ , for $\mu=-1$}
		\label{fig:pic2}
	\end{figure}
	
	Recall that both $Q,Q^3\subset Ker[\tilde{\cl}_+]^\perp$ are bell-shaped, hence the function
	$$
	\zeta(x):=Q(x)(Q^2(x)-Q^2(x_0))\perp Ker[\tilde{\cl}_+].
	$$
	satisfies  $\zeta(x)>0, x\in (-x_0, x_0)$, $\zeta(x)<0, x\in (-1, -x_0)\cup (x_0, 1)$. Thus, $\dpr{\zeta}{\psi}>0$, while $\zeta\perp Ker[\tilde{\cl}_+]$. A contradiction is reached.
	\subsubsection{Proof of $n(\tilde{\cl}_+)=1$}
	In this section, we show the remaining claim in \eqref{d:12}, namely that $\tilde{\cl}_+$ has exactly one negative eigenvalue. Note that it suffices to prove $n(\cl_+)=1$, as the operators $\cl_+, \tilde{\cl}_+$ have the same Morse index.
	
	To that end, note that Proposition \ref{prop:43} provides bell-shaped solutions for the profile equation \eqref{30}, with the property $n(\cl_+)=1$. We claim that at least one of the constrained minimizers in Proposition \ref{prop:43}  is actually in the form of one of the solutions in Proposition \ref{prop:43}. Going back to the full description of all possible  bell-shaped  solutions of \eqref{30}, in Proposition \ref{prop:10}, we see that the only other bell-shaped  solutions are in the form $\vp(0)=\vp_3, \vp(-T)=\vp(T)=0$ and $\vp_2=\bar{\vp}_1, \vp_3>0$.
	
	Now, fix  $c<0$.   Also, select and fix sufficiently large half-period $T>\f{1}{2|c|}\int_0^1 \f{dx}{\sqrt{x(1-x)}} dx$ and sufficiently small $L^2$ norm $\eps=\|v\|_{L^2}<<1$. Then, Proposition \ref{prop:43} guarantees a bell-shaped  solution $V$, which in particular satisfies the Euler-Lagrange equation \eqref{EL:10}. {\it We claim that $V$ is not of the form $V(0)=\vp_3, V(T)=V(-T)=0$. Once this is proven, we are done, since $V$ is then necessarily in the form of Proposition \ref{prop:50} and moreover, the corresponding linearized operator $L_+$ has exactly one negative eigenvalue.}
	
	Assume for a contradiction, that $V$ is of the form $V(0)=\vp_3, V(-T)=V(T)=0$.
	The difficulty with generating the solutions as constrained minimizers, as we just did is that we have no good control of the Lagrange multiplier $\zeta$, nor of the integration constant $a$. Instead, we have the parameters $T$ and $\|V\|_{L^2[-T,T]}^2$ to work with. Let us write the relations for the roots that we know. By the Viet's formulas we have for $\vp_1, \vp_2=\bar{\vp}_1$,
	$$
	\vp_1+\vp_2=4c - \vp_3
	$$
	Thus, we may compute
	\begin{eqnarray}
	\label{k:10}
	T &=& \int_0^T d\vp= \int_0^{\vp_3} \f{1}{\sqrt{\vp(\vp_3-\vp)
			(\vp-\vp_1)(\vp-\bar{\vp}_1)}} d\vp.
	\end{eqnarray}
	But   since the roots
	$\vp_1, \vp_2$ are complex-conjugate  and $\vp\geq 0$
	$$
	(\vp-\vp_1)(\vp-\bar{\vp}_1)\geq (\vp-\f{\vp_1+\bar{\vp}_1}{2})^2= (\vp-\f{4c-\vp_3}{2})^2\geq 4c^2.
	$$
	It follows that
	$$
	T\leq \f{1}{2|c|} \int_0^{\vp_3}  \f{1}{\sqrt{\vp(\vp_3-\vp)}} d\vp=\f{1}{2c} \int_0^1 \f{1}{\sqrt{x(1-x)}} dx.
	$$
	This is a contradiction with the choice of $T\geq \f{1}{2c} \int_0^1 \f{1}{\sqrt{x(1-x)}} dx. $.  It follows that $V$ is of the form of Proposition \ref{prop:50} and $n(\cl_+)=1$.
\end{proof}

\section{Spectral properties of $\cl$}
\label{sec:4}
In this section, we tackle the spectral properties of the self-adjoint  operator $\cl$. This operator, due to its matrix structure is naturally harder to analyze than its scalar counterparts, $\cl_+, \cl_-$. It turns out that it is possible to extract all the necessary spectral information, from the property
\begin{equation}
\label{assump}
\la_0(\cl_+)<\la_1(\cl_+)=0<\la_2(\cl_+); \ \ Ker(\cl_+)=span[\phi'],
\end{equation}
 which was established in Proposition \ref{prop:23}.

For $U=(u_1, u_2)$, where $u_1$ and $u_2$ are periodic functions with fundamental period $2 T$, we have
\begin{equation}
\label{60}
\langle \cl U, U\rangle = \langle \cl_1u_1,u_1\rangle +\langle M u_2,u_1\rangle+
\langle M^* u_1,u_2\rangle+\langle \cl_2u_2,u_2\rangle
\end{equation}
Let $u_2=\phi \tilde{u}_2$. After integrating by parts, we obtain
\begin{eqnarray*}
	\langle \cl_2u_2, u_2\rangle &=& \langle \phi \p_y \tilde{u}_2,
	\phi \p_y \tilde{u}_2 \rangle \\
	\langle M u_2, u_1\rangle + \langle M^* u_1, u_2\rangle &=& \langle \phi^{3}\p_y \tilde{u}_2, u_1\rangle
\end{eqnarray*}
Note that $\cl_1=\cl_+ +\f{1}{4}\phi^{4}$. Hence,
\begin{equation}
\label{70}
\langle \cl U, U\rangle = \langle \cl_+u_1,u_1\rangle +
\int_{-T}^{T}{\left[ \f{1}{2}\phi^{2}u_1+\phi (\phi^{-1}u_2)_y\right]^2}dy .
\end{equation}
 This particular nice structure of the bilinear form $\dpr{\cl U}{U}$, is not present, if one considers the linearized operator in the form \eqref{169a}. More specifically, there is an extra, sign-indefinite term in \eqref{70}, namely $const. \langle \phi^{3}\p_y \tilde{u}_2, u_1\rangle$ which complicates the analysis of $\cl$.   This is why we needed to convert to the form \eqref{169}.
Concerning the kernel  of $\cl$, we have the following result.
\subsection{Description of $Ker(\cl)$}
\label{sec:4.1}
\begin{proposition}
	\label{le:10}
	The operator $\cl$ has at most one negative eigenvalue, i.e. $n(\cl)\leq 1$.
	Next, \\ $dim(Ker(\cl))=2$. In fact,
	$$
	Ker(\cl)=span\{\Psi_1, \Psi_2\}, \Psi_1=\left(\begin{array}{c} 0 \\ \phi\end{array}\right),
	\Psi_2 = \left(\begin{array}{c} \phi' \\  -\f{1}{4} \phi^{3} \end{array}\right).
	$$
\end{proposition}
{\bf Remark:}  We will show later that, as expected,  $n(\cl)=1$.  Unfortunately, this cannot be done directly and requires some implicit analysis later on, see Section \ref{sec:5.2a} below.
\begin{proof}
	Based on the formula \eqref{70} and  \eqref{assump},  we see that if $u_1\perp \chi_0$ (where $\chi_0$ is the ground state for $\cl_+$), then $\cl|_{\{\chi_0\}^\perp}\geq 0$. Thus,
	$$
	\inf\limits_{U\perp \left(\begin{array}{c}\chi_0 \\ 0 \end{array}\right) } \dpr{\cl U}{U}\geq 	\inf\limits_{u_1\perp\chi_0} \dpr{\cl_+u_1}{u_1}\geq 0.
	$$
	This implies that $n(\cl)\leq 1$, as announced. We now discuss the structure of $Ker(\cl)$.
	
	It is convenient to split $L^2\times L^2= L^2_{even}\times L^2_{odd} \oplus L^2_{odd}\times L^2_{even}=:X_{e,o}\oplus X_{o,e}$ and
	since $\cl$ acts invariantly on those subspaces, it suffices to determine $Ker(\cl)$ on each.
	
	Let us first show that $X_{o, e}=span[\Psi_1, \Psi_2]$. Indeed, for $u_1\in L^2_{odd}, u_2\in L^2_{even}$, we have that $u_1\perp \chi_0$, whence by the argument above, for $U=\left(\begin{array}{c} u_1 \\ u_2\end{array}\right)\in X_{o,e}\cap Ker(\cl)$,
	$$
0= \dpr{\cl U}{U}\geq 	  \dpr{\cl_+u_1}{u_1}\geq 0.
	$$	
Recalling again that $u_1\perp \chi_0$, this implies that $u_1=0$ or $u_1=c \phi'$. From \eqref{70}, we have that the integral is zero as well, whence either $u_2=const. \phi$, when $u_1=0$ or else $u_2=-\f{c}{4} \phi^3$, when $u_1=c\phi'$.  This completes the analysis on $X_{o, e}$.

Next, we show that $Ker(\cl)\cap X_{e,o}=\{0\}$. This is a bit more complicated. Let $U=\left(\begin{array}{c} f \\ g \end{array}\right)\in X_{e,0}\cap Ker(\cl)$. We set
	\begin{equation}\label{l1.3}
	\left| \begin{array}{ll}
	\cl_1f+Mg=0 \\
	\\
	M^*f+\cl_2 g=0.
	\end{array} \right.
	\end{equation}
	Not that $\dpr{M^*f}{\phi}=\dpr{f}{M\phi}=0$ and  $Ker[\cl_2]=Ker[\cl_-]=span[\phi]$, so the second equation in (\ref{l1.3}) is solvable and in fact
	\begin{equation}
	\label{l1.4}
	g=-\cl_2^{-1}M^*f.
	\end{equation}
	Next, we will  be constructing Green function for the operator $\cl_2^{-1}=\cl_-^{-1}$.
	We have $\cl_2[\phi]=0$. The normalized function
	$$
	\psi(x)=\phi(x)\int_0^{x}{\frac{1}{\phi^2(s)}}ds, \; \; \left| \begin{array}{cc} \phi& \psi \\ \phi' & \psi'\end{array}\right|=1
	$$
	also solves  $\cl_2\psi=0$. The Green function, for an even function $f$  is represented by
	$$
	\cl_2^{-1}M^*f(x)=\phi(x)\int_{0}^{x}{\psi(s)M^* f(s)}ds-\psi(x)\int_{0}^{x}{\phi(s)M^*f(s)}ds+C_{M^*f}\psi(x),
	$$
	where $C_{M^*f}$ is a constant to be selected, so that $\cl_2^{-1}f$ is periodic with same period as $\phi$.  Integrating by parts yields
	\begin{equation}
	\label{96}
	\int_{0}^{x}\phi(s)  M^*f(s) ds=-\frac{1}{2}\phi^{3}f+
	\frac{1}{2}\phi^{3}(0)f(0)
	\end{equation}
	and
	\begin{equation}\label{96a}
	\int_{0}^{x}{\psi(s) M^*f(s)}ds=-\frac{1}{2}\phi^2\psi f+\frac{1}{2}\int_{0}^{x}{\phi f}.
	\end{equation}
	Also,
	\begin{eqnarray*}
		C_{M^*f} &=& -\frac{\phi(T)}{\psi(T)}\int_{0}^{T}{\psi M^*f}+\int_{0}^{T}{\phi M^*f}=-\frac{1}{2}\frac{1}{\int_{0}^{T}{\frac{1}{\phi^2}}dx}\int_{0}^{T}{\phi f}+
		\frac{1}{2}\phi^{3}(0)f(0)\\
		&=& -  \f{d_f}{2 d_1}+	\frac{1}{2}\phi^{3}(0)f(0).
	\end{eqnarray*}
	where
	$$
	d_1=\frac{1}{\int_{0}^{T}{\frac{1}{\phi^2}}dx}, \; \; d_f=\int_{0}^{T} \phi f.
	$$
	All in all,
	\begin{equation}
	\label{012}
	g=\f{d_f}{2d_1}\psi -\f{\phi}{2} \int_0^x \phi f
	\end{equation}
	Clearly, the formula for $g$ cannot be complete, without finding $f$, so we take this on now.
	Plugging \eqref{l1.4}   in the first equation of \eqref{l1.3} results in
	\begin{equation}
	\label{l1.5}
	\cl_1 f - M \cl_2^{-1}M^*f=0
	\end{equation}
	After some algebraic manipulations, we obtain
	\begin{eqnarray*}
		M \cl_2^{-1}M^*f &=& \frac{1}{4}\phi^{4}f-
		\frac{1}{4}\phi^{3}(0)f(0)\phi (x)+
		\frac{1}{2}C_{M^*f}\phi(x)\\
		\\
		&=&\  \f{\phi^{4}}{4}f-\f{d_1 d_f}{4} \phi.
	\end{eqnarray*}
	Finally, the left-hand side of   (\ref{l1.5}) is now in the form
	$$
	\cl_1 f-M \cl_2^{-1}M^*f=\cl_+f+\frac{1}{4}  d_1 d_f\phi,
	$$
	so, the equation to be solved is  $\cl_+f=-\frac{1}{4}  d_1 d_f\phi $. This equation has  a solution, since $Ker(\cl_+)=span[\phi']$. We obtain,
	\begin{equation}\label{l1.8}
	f=-\f{d_1 d_f}{4}\cl_+^{-1}[\phi].
	\end{equation}
{\it This still does not mean that we have found an element in $Ker(\cl)\cap X_{e,o}$, it simply means that if there is one, it must be in the form \eqref{l1.8}.} There is still a consistency condition to be satisfied, namely about $d_f$. We take dot product with $\phi$. We obtain
$
2 d_f=\dpr{f}{\phi}=-\f{d_1 d_f}{4} \dpr{\cl_+^{-1} \phi}{\phi},
$
	a relation that must be satisfied, in order to have a non-trivial element in $Ker(\cl)\cap X_{e,o}$. This solvability condition amounts to
	\begin{equation}
	\label{s:12}
		d_f(8 + d_1 \dpr{\cl_+^{-1} \phi}{\phi})=0.
	\end{equation}
	Clearly, if $d_f=0$, $f$   is trivial and this is not a new element of $Ker(\cl)$.  So, it remains that \\ $8 + d_1 \dpr{\cl_+^{-1} \phi}{\phi}=0$. Equivalently, it must be that $4 \int_{-T}^T \f{1}{\phi^2} + \dpr{\cl_+^{-1} \phi}{\phi}=0$. We  however need this quantity later on, so we have computed it: see \eqref{k:68} for $\dpr{\cl_+^{-1} \phi}{\phi}$ as well as the formula
	$$
	  \int_{-T}^T \f{1}{\phi^2} =  2g \int_{-K(\ka)}^{K(\ka)} \f{1}{\vp(2g \xi)} d\xi= \f{4g}{\vp_1\vp_3}
	  \left(\vp_3 K(\ka)+(\vp_1-\vp_3) \Pi\left[\f{\vp_1(\vp_3-\vp_2)}{\vp_3(\vp_1-\vp_2)}, \ka\right]\right),
	$$
	which we have done symbolically  in \textsc{Mathematica}. As a result, we can display the following pictures, Figure \ref{fig:Ker1} and  Figure \ref{fig:Kerm1},  which show that $4 \int_{-T}^T \f{1}{\phi^2} + \dpr{\cl_+^{-1} \phi}{\phi}>0$,
	\begin{figure}
		\centering
		\includegraphics[width=0.7\linewidth]{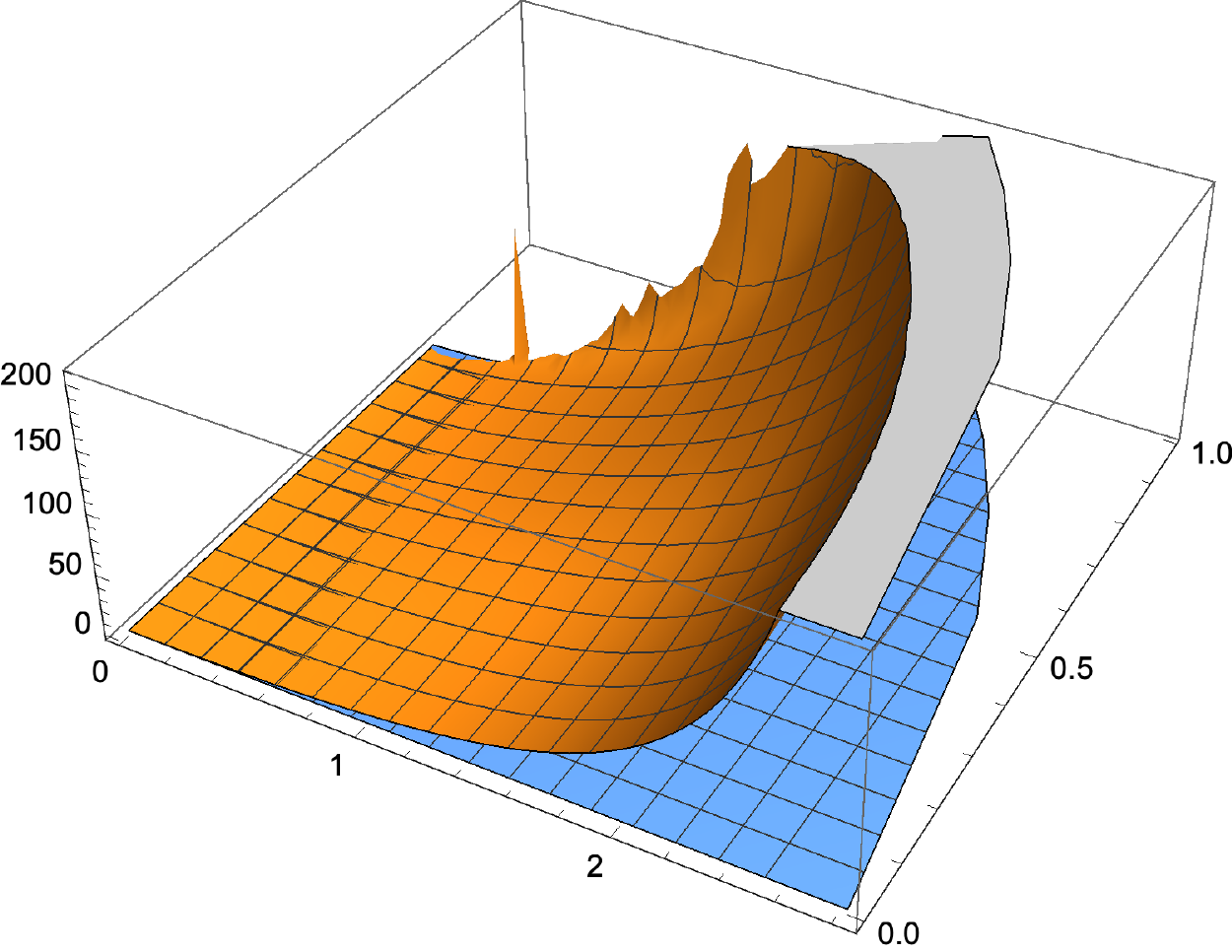}
		\caption{Graph of 	$4 \int_{-T}^T \f{1}{\phi^2}+\dpr{\cl_+^{-1}\phi}{\phi}$  , for $\mu=1$}
		\label{fig:Ker1}
	\end{figure}
	\begin{figure}
		\centering
		\includegraphics[width=0.7\linewidth]{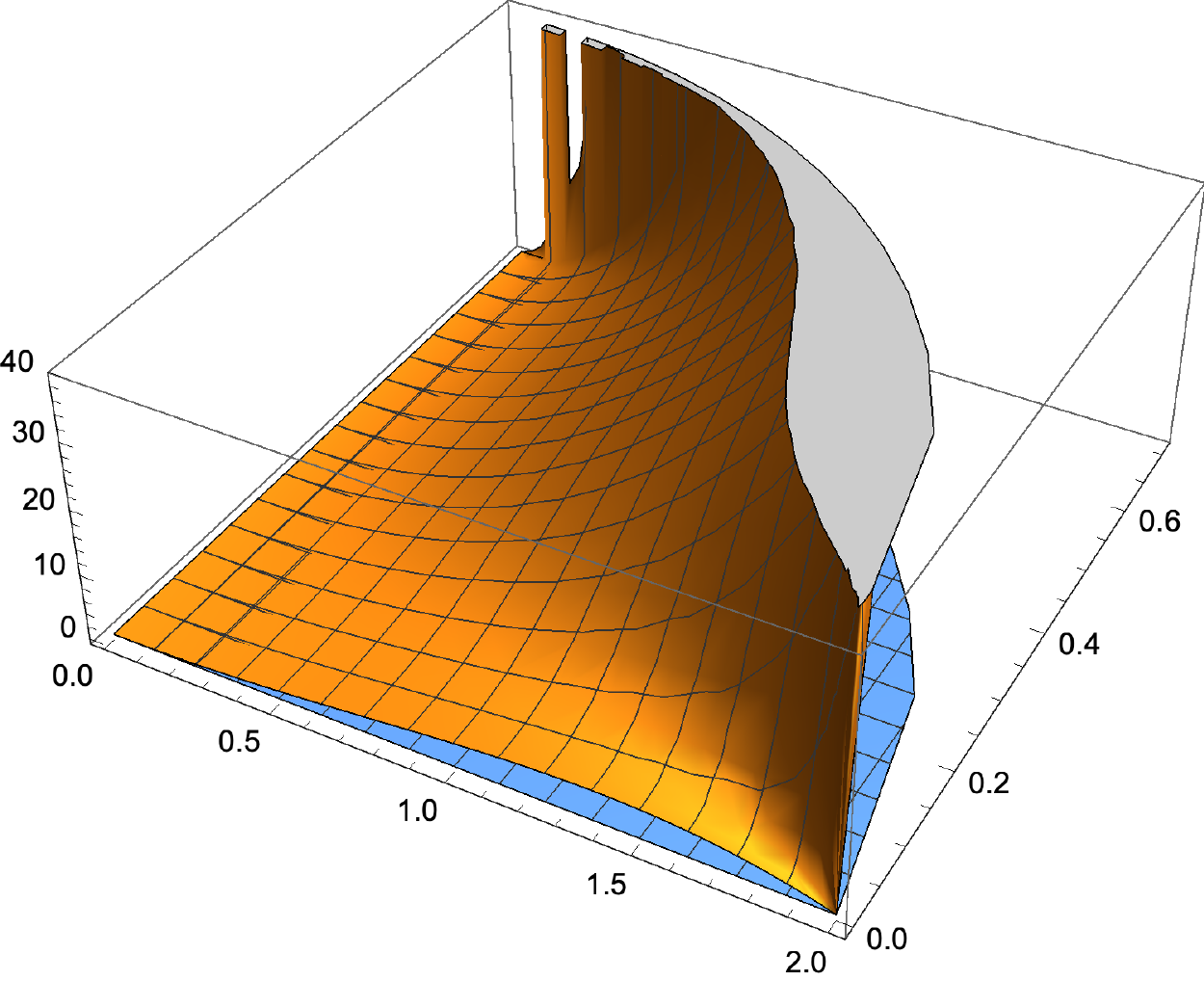}
		\caption{Graph of 	$4 \int_{-T}^T \f{1}{\phi^2}+\dpr{\cl_+^{-1}\phi}{\phi}$  , for $\mu=-1$}
		\label{fig:Kerm1}
	\end{figure}
	This finishes the verification that $Ker(\cl)\cap X_{e,o}=\{0\}$.
\end{proof}

\subsection{Preliminary calculations for the matrix $D$}
\label{sec:4.2}
According to the setup described in \eqref{dij}, we setup the matrix $D$ as follows
$$
D=\left(\begin{array}{cc}
\dpr{\cl^{-1} \cj \Psi_1}{\cj \Psi_1} & \dpr{\cl^{-1} \cj\Psi_1}{\cj \Psi_2} \\
\dpr{\cl^{-1} \cj\Psi_2}{\cj \Psi_1} & \dpr{\cl^{-1} \cj \Psi_2}{\cj \Psi_2}
\end{array}
\right),
$$
where $\Psi_1, \Psi_2$ span $Ker(\cl)$, they are  described in Proposition \ref{le:10}. More explicitly,
\begin{eqnarray*}
	D_{11}=\dpr{\cl^{-1} \cj \Psi_1}{\cj \Psi_1}  &=&
	\dpr{\cl^{-1} \left(\begin{array}{c} \phi \\ 0 \end{array}\right)}{\left(\begin{array}{c} \phi \\ 0 \end{array}\right)} , \\
	D_{1 2}=D_{2 1} = \dpr{\cl^{-1} \cj\Psi_1}{\cj \Psi_2} &=& \dpr{\cl^{-1} \cj\Psi_2}{\cj \Psi_1} = \dpr{\cl^{-1} \left(\begin{array}{c} \phi \\ 0 \end{array}\right)}{
		\left(\begin{array}{c} \f{\phi^{3}}{4} \\ \phi'\end{array}\right)}, \\
	D_{2 2} = \dpr{\cl^{-1} \cj \Psi_2}{\cj \Psi_2} &=&
	\dpr{\cl^{-1} \left(\begin{array}{c}  \f{\phi^{3}}{4}, \\ \phi'\end{array}\right) }{
		\left(\begin{array}{c}  \f{\phi^{3}}{4} \\ \phi'\end{array}\right)},
\end{eqnarray*}
where the normalization constant satisfies $k\int_0^L \phi^2 =
\f{1}{4} \int_0^L \phi^{4}$. We observe that since
$$
\left(\begin{array}{c} \phi \\ 0 \end{array}\right), \left(\begin{array}{c}  \f{\phi^{3}}{4} \\ \phi'\end{array}\right)\perp span[\Psi_1, \Psi_2]=Ker[\cl]
$$
we have that $\cj:Ker(\cl)\to Ker(\cl)^\perp$, so it is justified to take $\cl^{-1}$ in the above formulas.

As one can imagine, these quantities are quite hard to compute in general, especially with the involvement of the matrix Schr\"odinger operator $\cl$.    In fact, we have the following proposition, which establishes a reduced sufficient condition for stability of the waves $\phi$.
\begin{proposition}
	\label{prop:32}
Assume \eqref{assump}.	If
$$
D_{11}=\dpr{\cl^{-1} \left(\begin{array}{c} \phi \\ 0 \end{array}\right)}{\left(\begin{array}{c} \phi \\ 0 \end{array}\right)}<0,
$$
then $n(D)=1$ and the corresponding wave $\phi$ is stable.
 In fact, we have the formula
	\begin{equation}
	\label{100}
	D_{11}= \f{4\int_{-T}^T \f{1}{\phi^2}}{4\int_{-T}^T \f{1}{\phi^2}+ \dpr{\cl_+^{-1}[\phi]}{\phi}} \dpr{\cl_+^{-1}[\phi]}{\phi}
	\end{equation}
\end{proposition}
{\bf Remark:}  Note that the quantity in the denominator is positive, as established in the course of the proof of Proposition \ref{le:10}, see also Figures \ref{fig:Ker1} and \ref{fig:Kerm1}, which graphically confirm this. Thus, for all parameter values
$$
sgn(D_{11})=sgn(\dpr{\cl_+^{-1}[\phi]}{\phi})
$$
\begin{proof}
	First, if $\dpr{\cl^{-1} \left(\begin{array}{c} \phi \\ 0 \end{array}\right)}{\left(\begin{array}{c} \phi \\ 0 \end{array}\right)}<0$, it follows from the min-max principle that $\cl^{-1}$ has a negative eigenvalue, which implies that $\cl$ has one too, hence $n(\cl)\geq 1$.  Since, we have already established, in Proposition \ref{le:10},  that $n(\cl)\leq 1$, it would follows that $n(\cl)=1$. Furthermore, for the  matrix $D\in M_{2,2}$, we have $D_{11}=\dpr{D e_1}{e_1}<0$ means that $D$ too has a negative eigenvalue. Thus, $n(D)\geq 1$. Since we already know that $n(\cl)=1$, formula \eqref{20} implies that $n(D)\leq 1$, so $n(D)=1$ and hence, we have stability, from \eqref{e:20}. Thus, $D_{11}<0$ is sufficient  for stability.
	
	We now take on the question for the actual computation of  $D_{11}$. Despite being arguably the easiest entry in the matrix $D$ to calculate, it  is not an easy task to actually compute it. Its analysis is  related to the analysis of $Ker(\cl)$ in Proposition \ref{le:10}. Let
	\begin{equation}
	\label{93a}
	\cl\left(\begin{array}{ll} f\\ g \end{array}\right)=\left(\begin{array}{ll} \phi\\ 0 \end{array}\right),
	\end{equation}
which as we have observed is solvable in $X_{e,o}$, due to the fact that $\left(\begin{array}{ll} \phi\\ 0 \end{array}\right)\perp Ker[\cl]$.
We note that such a solution comes with  the property $\left(\begin{array}{ll} f\\ g \end{array}\right)\in Ker(\cl)^\perp$.
	So, \eqref{93a} is equivalent to
	\begin{equation}\label{l1.3a}
	\left| \begin{array}{ll}
	\cl_1f+Mg=\phi \\
	\\
	M^*f+\cl_2 g=0.
	\end{array} \right.
	\end{equation}
	Proceeding as in the proof of Proposition \ref{le:10},
	\begin{equation}
	\label{l1.5a}
	\cl_1 f - M \cl_2^{-1}M^*f=\phi .
	\end{equation}
	 resulting in
   \begin{equation}
   \label{l1.8a}
   \cl_+f+\frac{1}{4}  d_1 d_f\phi= \phi,
   \end{equation}
   where
   $$
   d_1=\frac{1}{\int_{0}^{T}{\frac{1}{\phi^2}}dx}, \; \; d_f=\int_{0}^{T} \phi f.
   $$
   We obtain,
   \begin{equation}\label{l1.9}
   f=(1-\f{d_1 d_f}{4}) \cl_+^{-1}[\phi].
   \end{equation}

   Next, we determine $d_f$. We simply take dot product of \eqref{l1.9} with $\phi$. We obtain the equation
   $$
   2 d_f=(1-\f{d_1 d_f}{4}) \dpr{\cl_+^{-1} \phi}{\phi}.
   $$
   Let us note that this equation must have solutions as \eqref{l1.3} does have a solution. In particular,
   \begin{equation}
   \label{97}
   d_f=\f{\dpr{\cl_+^{-1} \phi}{\phi}}{2+\frac{d_1}{4}  \dpr{\cl_+^{-1}\phi}{\phi}},
   \end{equation}
Since by definition, $D_{11}=2 d_f$, we arrive at the formula \eqref{100}.
\end{proof}
It becomes clear that in order to check for the stability, we need to be able to calculate various quantities like $\dpr{\cl_+^{-1} \phi}{\phi}$. We have already computed that, subject to rescaling, see \eqref{j:22}.
\section{Analysis of the spectral stability for the bell-shaped waves of DNLS}
\label{sec:5}
In this section, we use our preliminary calculations, which allow us to compute the various quantities involved in the matrix $D$.

\subsection{Computing $\dpr{\cl_+^{-1}\phi}{\phi}$}
\label{sec:5.1}
In the calculations for $D_{11}$, see \eqref{100}, a major role is played by $\dpr{\cl_+^{-1}\phi}{\phi}$. In order to compute that, we use the formula \eqref{j:22}, which gives $\tilde{L}^{-1} Q$, in the rescaled framework of Section \ref{sec:3.2}.

So, let us continue to use the setup introduced in Section \ref{sec:3.2}
and more precisely in the equation  \eqref{d:10}. By taking dot product of \eqref{j:22} with $Q$, we obtain
\begin{eqnarray*}
	\dpr{\tilde{\cl}_+^{-1} Q}{Q} &=& 16 T^{-2} \f{\f{c_\mu}{2} \dpr{Q}{Q} +c_\mu \f{K(\ka)}{K'(\ka)} \dpr{Q_\ka}{Q} -\left(c+c_\ka \f{K(\ka)}{K'(\ka)}\right)\dpr{Q_\mu}{Q}}{c+c_\ka \f{K(\ka)}{K'(\ka)} - 2\mu c_\mu},
\end{eqnarray*}
where recall that the function $c=c(A,B,C)$ is given explicitly in \eqref{r:51}, while the quantities $A,B,C$, all in terms of $g,\ka, \mu$ are explicitly in \eqref{r:10}, \eqref{r:20}, \eqref{r:30}.

We have, $\dpr{Q}{Q}=T^{-1} \|\phi\|^2$. Also, since $T=T(g, \ka)$ is independent on $\mu$,
$$
\dpr{Q_\mu}{Q}=\int_{-1}^1 Q_\mu(\xi) Q(\xi) d\xi=\f{1}{2}  \p_\mu \int_{-1}^1 Q^2(\xi)d\xi= \f{1}{2}  \p_\mu T^{-1} \|\phi\|^2=\f{T^{-1}}{2} \p_\mu \|\phi\|^2
$$
On the other hand,
$$
\dpr{Q_\ka}{Q}=\int_{-1}^1 Q_\ka(\xi) Q(\xi) d\xi=\f{1}{2}  \p_\ka \int_{-1}^1 Q^2(\xi)d\xi= \f{1}{2}  \p_\ka[ T^{-1} \|\phi\|^2]=\f{T^{-1}}{2} \p_\ka \|\phi\|^2 - \f{T^{-2} T_\ka}{2} \|\phi\|^2.
$$
Thus, we have reduced matters to computing the following formula
\begin{eqnarray*}
	\nonumber
	\dpr{\tilde{\cl}_+^{-1} Q}{Q} &=&  16 T^{-2} \f{\f{c_\mu}{2} T^{-1} \|\phi\|^2  +c_\mu \f{K(\ka)}{K'(\ka)} (\f{T^{-1}}{2} \p_\ka \|\phi\|^2 - \f{T^{-2} T_\ka}{2} \|\phi\|^2) -\left(c+c_\ka \f{K(\ka)}{K'(\ka)}\right)\f{T^{-1}}{2} \p_\mu \|\phi\|^2}{c+c_\ka \f{K(\ka)}{K'(\ka)} - 2\mu c_\mu}\\
	&=& 16 T^{-3}  \f{\f{c_\mu}{2}  \|\phi\|^2  +c_\mu \f{K(\ka)}{K'(\ka)} (\f{1}{2} \p_\ka \|\phi\|^2 - \f{K'(\ka)}{2 K(\ka)} \|\phi\|^2) -\left(c+c_\ka \f{K(\ka)}{K'(\ka)}\right)\f{1}{2} \p_\mu \|\phi\|^2}{c+c_\ka \f{K(\ka)}{K'(\ka)} - 2\mu c_\mu}=\\
	&=& 8 T^{-3}  \f{c_\mu \f{K(\ka)}{K'(\ka)}\p_\ka \|\phi\|^2   -\left(c+c_\ka \f{K(\ka)}{K'(\ka)}\right) \p_\mu \|\phi\|^2}{c+c_\ka \f{K(\ka)}{K'(\ka)} - 2\mu c_\mu}.
\end{eqnarray*}
Since $\dpr{\tilde{\cl}_+^{-1} Q}{Q}=T^{-3} \dpr{\cl_+^{-1} \phi}{\phi}$, we arrive at the formula
\begin{equation}
\label{k:68}
\dpr{\cl_+^{-1} \phi}{\phi}=8 \f{c_\mu \f{K(\ka)}{K'(\ka)}\p_\ka \|\phi\|^2   -\left(c+c_\ka \f{K(\ka)}{K'(\ka)}\right) \p_\mu \|\phi\|^2}{c+c_\ka \f{K(\ka)}{K'(\ka)} - 2\mu c_\mu}.
\end{equation}
As we saw earlier, the denominator is never zero, per explicit calculations done earlier, see Figure  \ref{fig:pic1} and Figure \ref{fig:pic2} for a particular slices at $\mu=1$, $\mu=-1$ respectively.

Thus, we need to evaluate $\|\phi\|_{L^2}^2$in terms of $g, \ka, \mu$.
Using \textsc{Mathematica}, we computed
\begin{eqnarray*}
	\|\phi\|_{L^2}^2 &=&   2g \int_{-K(\ka)}^{K(\ka)} \vp(2g \xi) d\xi= 4g\left(\vp_1 K(\ka)+(\vp_3-\vp_1) \Pi\left[\f{\vp_3-\vp_2}{\vp_1-\vp_2}, \ka\right]\right),
\end{eqnarray*}
where $\Pi$ is the elliptic $\Pi$ function.

With this formula in hand,
we compute the quantities in \eqref{k:68} using \textsc{Mathematica}. The results can be seen  in the slices of the graphs for $\dpr{\cl_+^{-1}\phi}{\phi}$, in Figure \ref{fig:pic4}, for $\mu=1$ and Figure \ref{fig:pic5} for $\mu=-1$. From these images (and this is the case for all values of $\mu$ that we have tried), the expression $\dpr{\cl_+^{-1}\phi}{\phi}$ always changes sign over the domain. In particular, there is a always a region $\tilde{\Om}$, where it takes negative values. It follows that $D_{1 1}$  vanishes on a curve in the domains, and it  takes positive and negative values as well.
\begin{figure}
	\centering
	\includegraphics[width=0.7\linewidth]{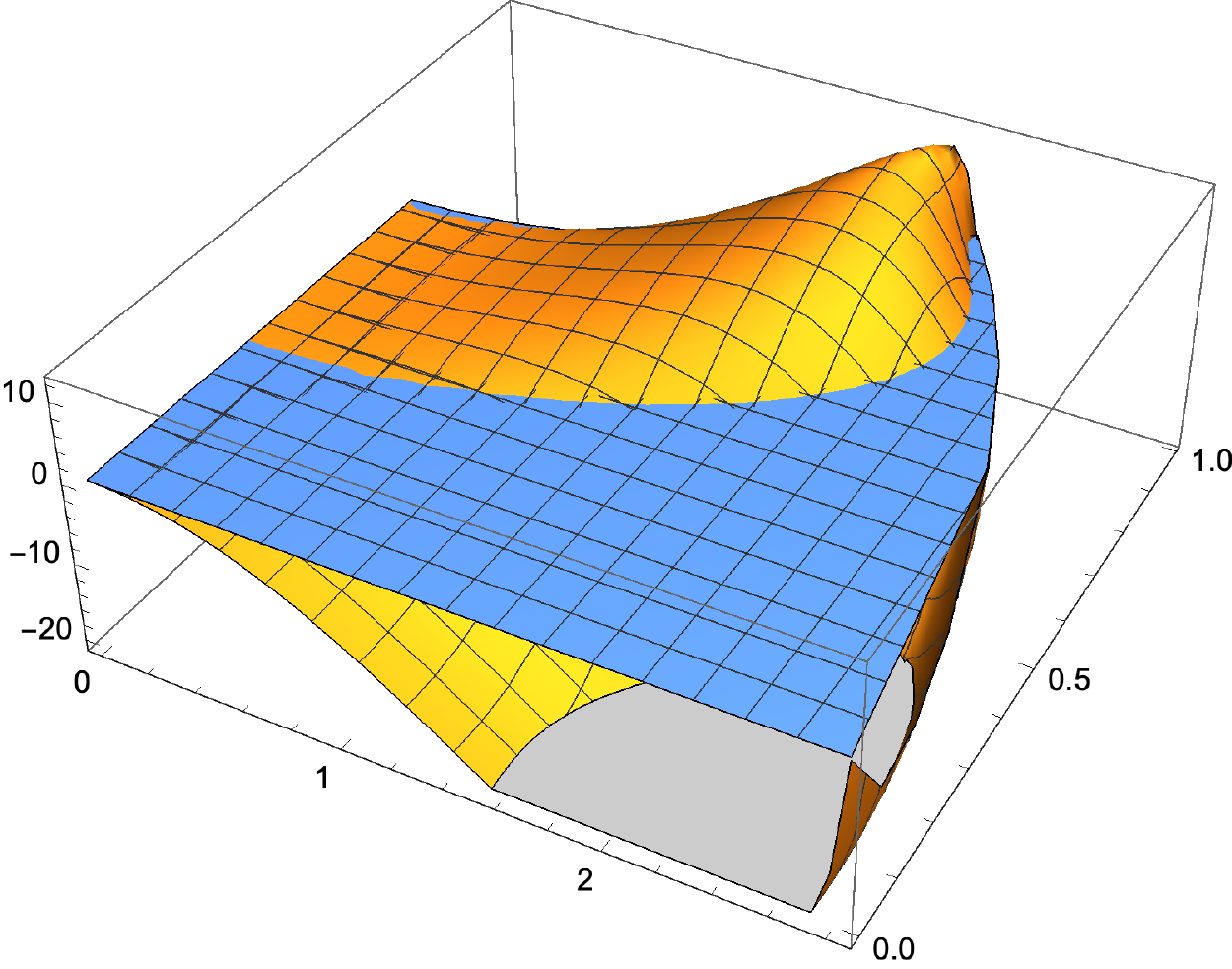}
	\caption{Graph of 	$\dpr{\cl_+^{-1}\phi}{\phi}$  , for $\mu=1$}
	\label{fig:pic4}
\end{figure}
\begin{figure}
	\centering
	\includegraphics[width=0.7\linewidth]{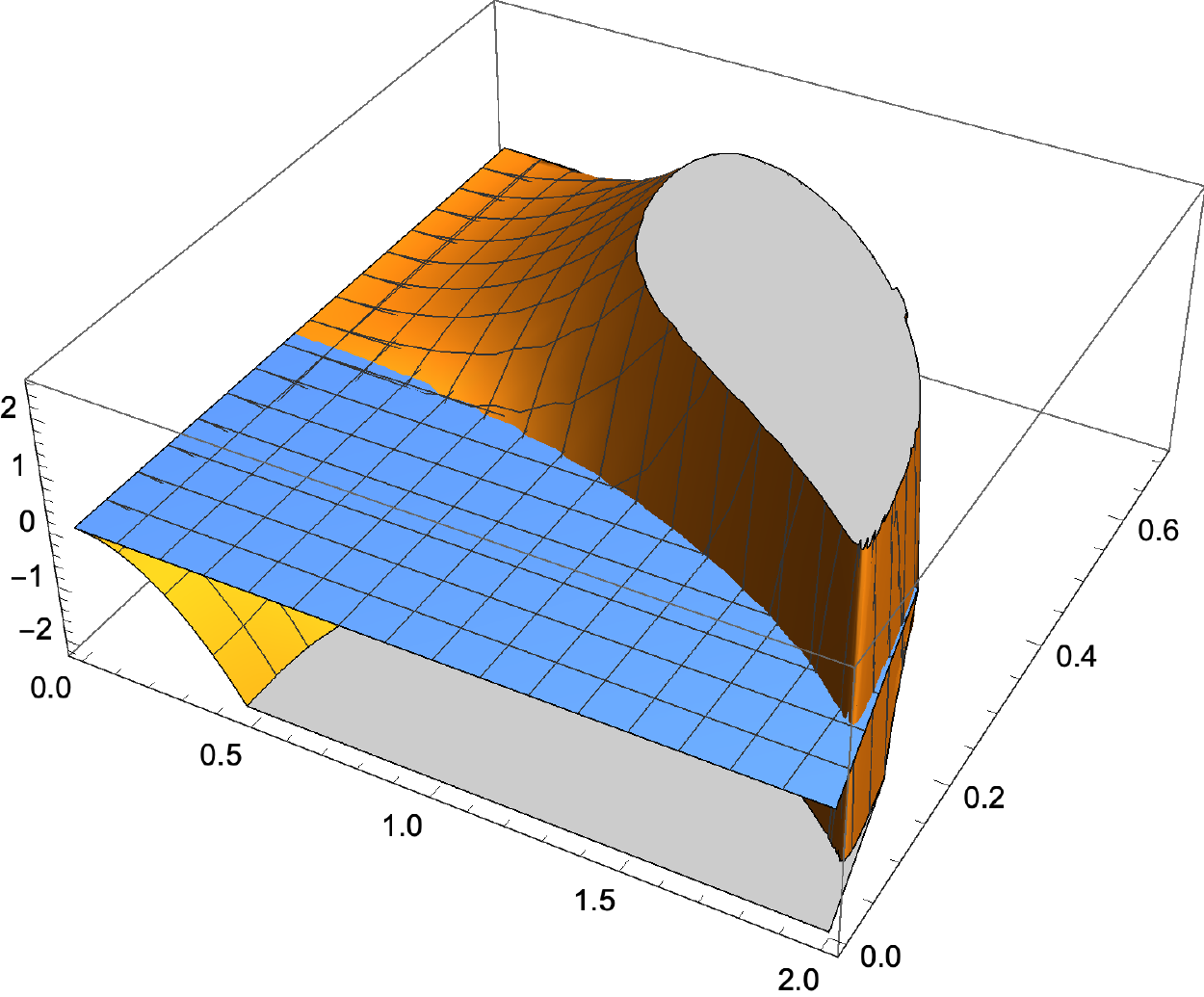}
	\caption{Graph of 	$\dpr{\cl_+^{-1}\phi}{\phi}$  , for $\mu=-1$}
	\label{fig:pic5}
\end{figure}
\subsection{$n(\cl)=1$ and the waves with $\dpr{\cl_+^{-1}\phi}{\phi}<0$ are spectrally  stable}
	\label{sec:5.2a}
In this section, we finally confirm  that $n(\cl)=1$. This is  achieved by piecing together several conclusions established in the previous sections.

We have already seen that $n(\cl)\leq 1$ in Proposition \ref{le:10}. We have   argued in the previous section that that the quantity
$\dpr{\cl_+^{-1}\phi}{\phi}$ is always negative, somewhere in the parameter domain, see also Figures \ref{fig:pic4}, \ref{fig:pic5}. Thus, according to \eqref{100}, $sgn(D_{11})=sgn(\dpr{\cl_+^{-1}\phi}{\phi})=-1$ on a portion of the domain. Thus, it follows that $n(D)\geq 1$, somewhere in the parameter domain. On the other hand, from the instability index theory, see \eqref{e:20}, we always have the inequality $n(\cl)\geq n(D)$, so $n(\cl)\geq 1$ somewhere on the parameter domain. We claim that $n(\cl)=1$ for all values in the parameter domain. Indeed, if $n(\cl)=1$ somewhere on it, a potential transition to
$n(\cl)=0$, due to the continuous dependence on the parameters, happens only if the negative eigenvalue crosses the zero en route to becoming a positive one. This would require, at least for some value of the parameters to have three vectors in $Ker(\cl)$. According to Proposition \ref{le:10}, this is not the case.  Thus, $n(\cl)=1$ for all parameters described in Proposition \ref{prop:50}. Hence, as discussed after \eqref{e:20}, the stability of the waves is equivalent to $n(D)=1$ or equivalently, $det(D)<0$.

In particular, since $n(D)\geq 1$, whenever $\dpr{\cl_+^{-1}\phi}{\phi}<0$, we have that $k_{Ham}=n(\cl)-n(D)=0$, hence spectral stability holds for these values.
\subsection{$D_{1 2}$ and $D_{1 1}$ do not vanish simultaneously: conclusion of the proof}
\label{sec:5.2}
We need to establish  that
$$
D_{1 2}=\dpr{f}{\f{\phi^3}{4}} + \dpr{g}{\phi'},
$$
with $f,g$ as introduced in \eqref{93a} does not vanish simultaneously with $D_{1 1}$.
Let us consider the expression for $D_{1 2}$, exactly on the set where $D_{1 1}=0$. Clearly $D_{1 1}=0$ exactly when $\dpr{\cl_+^{-1}\phi}{\phi}=0$ and since $f= const. \cl_+^{-1} \phi$, precisely when $\dpr{f}{\phi}=0$. Thus, on this set, we can check that  $g=-\f{\phi}{2} \int_0^x f\phi$.   But, on the set $\{D_{11}=0\}$, an integration by parts shows
$$
D_{12}=\dpr{f}{\f{\phi^3}{4}} + \dpr{g}{\phi'}=\f{1}{2} \dpr{f}{\phi^3}=const.  \dpr{\cl_+^{-1} \phi}{\phi^3}.
$$
Thus, it suffices to check $\dpr{\cl_+^{-1}\phi}{\phi}, \dpr{\cl_+^{-1} \phi}{\phi^3}$ do not vanish simultaneously. To this end, recall that $n(\cl_+)=1, Ker[\cl_+)=span[\phi']$, and  denote its positive ground state of $\cl_+$ by $\Psi_0$. Clearly, $\dpr{\phi}{\Psi_0}>0, \dpr{\phi^3}{\Psi_0}>0$. This, there is a scalar $c_0>0$, so that $\dpr{\phi-c_0\phi^3}{\Psi_0}=0$. It follows that
\begin{equation}
\label{nm}
\dpr{\cl_+^{-1} \phi}{\phi} - c_0 \dpr{\cl_+^{-1} \phi}{\phi^3}=
\dpr{\cl_+^{-1} \phi}{\phi-c_0\phi^3}=
\dpr{\cl_+^{-1} P_{\{\Psi_0, \phi'\}^\perp}  \phi}{P_{\{\Psi_0, \phi'\}^\perp} (\phi-c_0\phi^3)}>0,
\end{equation}
 because $P_{\{\Psi_0, \phi'\}^\perp}$ restricts to the positive subspace of $\cl_+^{-1}$. This last inequality \eqref{nm} shows that $\dpr{\cl_+^{-1}\phi}{\phi}, \dpr{\cl_+^{-1} \phi}{\phi^3}$ cannot not vanish simultaneously. }

\section{Stability analysis for the quintic NLS waves}
\label{sec:6}
The spectral problem \eqref{l:14} is easy to analyze, with the tools that we have prepared so far. Indeed, by Proposition \eqref{prop:23}, we have seen that $n(\cl_+)=1$, while $n(\cl_-)=0$. In addition, $Ker(\cl_+)=span[\phi']$, while $Ker(\cl_-)=span[\phi]$. Applying index counting theory (and more specifically \eqref{e:20}), we see that the matrix $D$ is one dimensional, namely $D_{11}=\dpr{\cl_+^{-1} \phi}{\phi}$. In addition,
$k_{Ham}=0$ (and hence the waves are stable) exactly when $\dpr{\cl_+^{-1} \phi}{\phi}<0$ and unstable, if \\ $\dpr{\cl_+^{-1} \phi}{\phi}>0$, with a change of instability (and an additional element in the generalized kernel in the Hamiltonian linearized operator \eqref{l:14} for $\dpr{\cl_+^{-1} \phi}{\phi}=0$. Thus, we need to find $\dpr{\cl_+^{-1} \phi}{\phi}$, for this new restricted set of parameters. We set on to describe the waves in the style of Proposition \ref{prop:50}.
\begin{proposition}
	\label{prop:51}
	Let
	\begin{equation}
	\label{384}
	0<g<\infty \ \ \& \ \ \ka\in (0,1) \ \ \& \ \ \mu=\f{4 \sqrt{1-k^2+k^4}}{g^2}
	\end{equation}
	Let the roots
	$\vp_1, \vp_2, \vp_3$ are as described in \eqref{r:50}. Then,  the two family of waves  (corresponding to $\pm$ values of $\mu$) described in   \eqref{2.52} are parametrized by $g,\ka$ are all non-vanishing bell-shaped waves of the quintic NLS.
\end{proposition}
\begin{proof}
	Recall that the waves \eqref{2.52} correspond to those of \eqref{2.5}, with the condition $c=0$. Solving in the formula for $c$ from \eqref{r:51}, $c=0$, we end up with the relation $AB+BC-AC=0$, which is solved in terms of $\mu$ to exactly $\mu=  \f{4 \sqrt{1-k^2+k^4}}{g^2}$. Finally, note that the constraint for $\mu$ ensures that the inequalities required in \eqref{382}, namely
	$
	g<\sqrt{\f{8}{\mu}}.
	$
	is satisfied.
	This completes the proof of Proposition \ref{prop:51}.
\end{proof}
Our next task is to determine the sign of the quantity $\dpr{\cl_+^{-1} \phi}{\phi}$, on the set of parameters outlined in the constraint \eqref{384}. We have plotted the graph of the relevant graph in Figure \ref{fig:pic9}. It shows almost perfect stability result for $0<\ka<0.54$ and instability for $\ka\in (0.54, 1)$.
\begin{figure}
	\centering
	\includegraphics[width=0.7\linewidth]{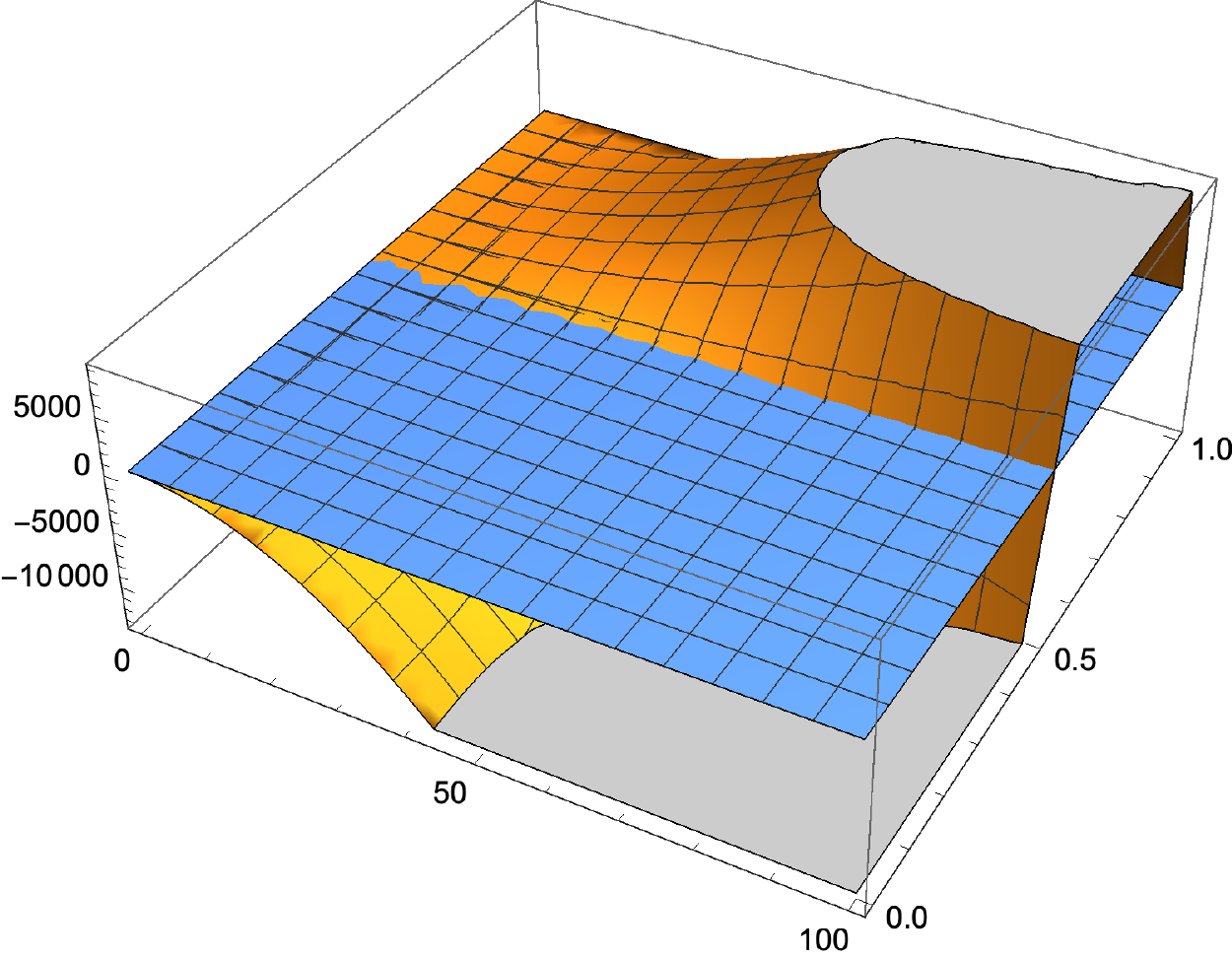}
	\caption{Graph of 	$\dpr{\cl_+^{-1}\phi}{\phi}$ for quintic NLS waves}
	\label{fig:pic9}
\end{figure}

\end{document}